\documentclass[11pt, oneside, reqno]{article}

\usepackage{amssymb}
\usepackage{amsmath}
\usepackage{amsthm}
\usepackage{amsfonts}
\usepackage{enumerate}
\usepackage{esint,subfig}
\usepackage{latexsym,amssymb}
\usepackage{color}
\usepackage[hidelinks]{hyperref}
\usepackage[latin1]{inputenc}
\usepackage{multicol}
\usepackage[pdftex]{graphicx}
\usepackage[nottoc,notlot,notlof]{tocbibind}

\usepackage{geometry,color,mathtools,fixltx2e,microtype,lastpage,mathrsfs}

\usepackage{xr}

\allowdisplaybreaks

\numberwithin{equation}{section}

\theoremstyle{plain}
\newtheorem{thm}{Theorem}[section]
\newtheorem{cor}[thm]{Corollary}
\newtheorem{lem}[thm]{Lemma}
\newtheorem{prop}[thm]{Proposition}

\theoremstyle{definition}
\newtheorem{defn}{Definition}[section]

\theoremstyle{remark}
\newtheorem{oss}{Remark}[section]

\def\R{\mathbb{R}}
\def\P{\mathbb{P}}
\def\C{\mathbb{C}}

\def\N{\mathbb{N}}
\def\NN{\mathcal{N}}
\def\NNN{\mathscr{N}}

\def\L{\mathscr{L} }
\def\A{\mathscr{A}}

\def\E{\mathbb{E}}

\def\to{\longrightarrow}
\def\sto{\rightarrow}

\def\1{\mathbf{1}}
\def\e{\varepsilon}

\DeclareMathOperator{\Cov}{Cov}
\DeclareMathOperator{\Var}{Var}

\def\area{\operatorname{area}}
\def\diam{\operatorname{diam}}

\makeatletter
\newcommand{\vast}{\bBigg@{3.5}}
\newcommand{\Vast}{\bBigg@{5}}
\makeatother
\def\paref#1{(\ref{#1})}

\renewcommand{\(}{\left(}
\renewcommand{\)}{\right)}
\renewcommand{\[}{\left[}
\renewcommand{\]}{\right]}
\renewcommand{\{}{\left\lbrace}
\renewcommand{\}}{\right\rbrace}

\newcommand{\norm}[1]{\left\lVert#1\right\rVert}
\newcommand{\abs}[1]{\left\lvert#1\right\rvert}

\newcommand{\eps}{\varepsilon}

\def\endmproof{\hfill \mbox{\raggedright \rule{0.1in}{0.1in}}}

\begin{document}

\title{\sc Gaussian Random Measures \\ Generated by Berry's Nodal Sets}

\author{{\large Giovanni Peccati* and Anna Vidotto**}\\
$\ $\\
\small{*\emph{Unit\'e de Recherche en Math\'ematiques, Universit\'e du Luxembourg}}\\ 
\small{**\emph{Dipartimento di Matematica, ``Tor Vergata" Universit\`a di Roma }}}  
\date{\small \today}

\maketitle

\begin{abstract}  

We consider vectors of random variables, obtained by restricting the length of the nodal set of Berry's random wave model to a finite collection of (possibly overlapping) smooth compact subsets of $\R^2$. Our main result shows that, as the energy diverges to infinity and after an adequate normalisation, these random elements converge in distribution to a Gaussian vector, whose covariance structure reproduces that of a homogeneous independently scattered random measure. A by-product of our analysis is that, when restricted to rectangles, the dominant chaotic projection of the nodal length field weakly converges to a standard Wiener sheet, in the Banach space of real-valued continuous mappings over a fixed compact set. An analogous study is performed for complex-valued random waves, in which case the nodal set is a locally finite collection of random points.

\medskip


\medskip


\end{abstract}

%\tableofcontents

%------------------------------------------------------------------------------------------------------------------------------------------------------------------------------------------------------------------------------------------------------

\allowdisplaybreaks
  
\section{Introduction}

The aim of this paper is to prove second order results for sequences of random vectors obtained by restricting the nodal length of {\bf Berry's random wave model} to finite collections of (possibly intersecting) smooth compact subsets of $\R^2$. Such a model was first introduced in \cite{Be:77}, and typically emerges as the local scaling limit of random fields on Riemannian surfaces that are approximately eigenfunction of the associated Laplace-Beltrami operator -- see e.g. \cite{Ze:09, CH:16}, as well as Section \ref{ss:introapp} below. Berry's model has been recently the object of a an intense study, mainly in connection with the high-frequency analysis of local and non-local geometric quantities associated with the nodal sets of smooth random fields --- see e.g. \cite{CH:16, BCW:17, KW:18} and the references therein.

\smallskip

Our main finding is that, in the high-energy limit, the above mentioned random vectors verify a multivariate central limit theorem (CLT), with a limiting covariance matrix reproducing the dependence structure of a homogeneous independently scattered random measure. Such a result extends the one-dimensional CLT recently proved in \cite{NPR:19} (see also \cite{Be:02} for a seminal contribution). An analogous analysis will be also realised for complex-valued random waves, whose nodal set is almost surely a locally finite collection of points --- see again \cite{Be:02, NPR:19}.

\smallskip

The contributions of the present paper are part of a growing body of research (see e.g. \cite{MPRW:16, CMW:16, CMW:16b, CM:19b, CM:18b, PR:18,  To:18b, To:18, DNPR:19, RW:18, Ca:19, BMW:19, MRW:19}) focussing on second order results for local quantities associated with nodal sets of Gaussian random waves, deduced by using tools of Gaussian analysis, in particular variational and Wiener chaos techniques. See \cite{Be:02, RW:08, ORW:08, Wi:10, KKW:13} for a sample of earlier fundamental contributions on variance estimates and related quantities.

\smallskip

\noindent{\bf Some conventions.} In what follows, every random object is defined on a common probability space $(\Omega, \mathscr{F}, \P)$, with $\E$ indicating mathematical expectation with respect to $\P$. The symbol $\Longrightarrow$ stands for convergence in distribution of random vectors (note that such a notation is silent on the dimension of the underlying objects). Given two positive sequences $\{a_n\}$, $\{b_n\}$, we write $a_n \sim b_n$ whenever $a_n/b_n\sto 1$, as $n\sto \infty$. When no further specification is provided, the lowercase letter $c$ is used to denote an absolute finite and positive constant, whose exact value might change from line to line.

\smallskip

\noindent{\bf Plan.} The paper is organized as follows: in Section \ref{ss:overview} we introduce the model and the main objects of our study, in Section \ref{ss:intromain} we present our main results, and in Section \ref{ss:introapp} we discuss some applications to monochromatic and superposed waves. Section \ref{proofs} contains the proofs while Appendix \ref{appA} contains some ancillary results.

%%%%%%%%%%%%%%%%%%%%%%%%%%%%%%%%%%%%%%%%%%%%%%%%%%%%%%%%%%%%%%%%%%%%%%%%%%%%%%%%%%%%%%%%%%%%%%%%%%%%%%%%%%%%%%%%%%%%%

\section{The model}\label{ss:overview}

For $E>0$, the  {\bf real-valued Berry's random wave model} \cite{Be:77, Be:02, NPR:19} with {\bf energy} $4\pi^2 E$, written as $$B_E = \left\lbrace B_E(x) : x\in \R^2\right\rbrace,$$ is defined as the centered Gaussian field on $\R^2$ having covariance kernel
\begin{equation}\label{covE}
r^E(x, y) = r^E(x-y) := J_0(2\pi \sqrt E \norm{x - y}), \,\,\, x,y\in \R^2\,.
\end{equation}
where $J_0$ indicates the Bessel function of the first kind with order $\alpha = 0$, namely
\begin{equation}\label{seriesJ0}
J_0(u) = \sum_{m=0}^{+\infty} \frac{(-1)^m}{(m!)^2}\left (\frac{u}{2}\right)^{2m},\qquad u\in \R.
\end{equation} 
Note that formula \eqref{covE} immediately yields that $B_E$ is {\bf isotropic}, that is: the distribution of $B_E$ is invariant with respect to rigid motions of the plane. It is a standard fact that $J_0$ is the unique radial solution to the equation 
\begin{equation}\label{e:1helmholtz}
\Delta f + f= 0
\end{equation}
verifying $f(0)=1$; here, $\Delta:=\partial^2/\partial x_1^2 + \partial^2/\partial x_2^2$ denotes as usual the Laplace operator.   

\smallskip

It is known (see e.g. \cite{NPR:19}) that $B_E$ can be represented as a random series 
\begin{equation}\label{serie}
B_E(x) = B_E(r,\theta)=\Re \left( \sum_{m=-\infty}^{+\infty} a_m J_{|m|}(2\pi \sqrt{E} r)\e^{im\theta} \right),
\end{equation}
where we have used polar coordinates $(r,\theta)=x$, $\Re(s)$ denotes the real part $s$, the set $\{a_m\}$ is a collection of i.i.d. complex Gaussian random variables such that $\E[a_m]=0$ and $\E[|a_m|^2]=2$, and $J_\alpha$ indicates the Bessel function of the first kind of index $\alpha$. The series \paref{serie} is almost surely  convergent, and moreover uniformly convergent on any compact set, and the sum is a real analytic function -- see again \cite{NPR:19} and the references therein. From the representation \paref{serie} one also infers  that $B_E$ is almost surely  an {\bf eigenfunction} of $\Delta$ with eigenvalue $-4\pi^2 E$, i.e.: with probabilitly 1, the random mapping $x\mapsto B_E(x)$ solves the Helmholtz equation
$$
\Delta B_E(x) + 4\pi^2 E \cdot B_E(x) = 0, \quad x\in \R^2.
$$
We will also consider a complex version of $B_E$ (referred to as the {\bf complex-valued Berry's random wave model} with energy $4\pi^2E$). Such a field is defined as
\begin{equation}\label{e:complexberry}
B^{\mathbb C}_E(x) := B_E(x) + i \widehat{B}_E(x), \quad x\in \R^2,
\end{equation}
 where $\widehat{B}_E$ is an independent copy of $B_E$. One easily checks that $B^{\mathbb C}_E$ almost surely  verifies the equation $\Delta B^{\mathbb C}_E +4\pi^2 E \cdot B^{\mathbb C}_E=0$. 
 
 \smallskip
 \begin{oss}\label{r:not}{\rm 
 In order to make more explicit the connection with \cite{Be:02, CH:16}, for $k>0$ we will sometimes use the special notation $b_k$ and $b^{\mathbb C}_k$, respectively, to indicate the fields $B_E(x)$ and $B^{\mathbb C}_E(x)$ in the special case $E = k^2/(4\pi^2)$. In particular, $b_k$ and $b^{\mathbb C}_k$ are isotropic Gaussian solutions of the equation $\Delta f+k^2 f = 0$.}
 \end{oss}
 \smallskip

The principal focus of our analysis are the two {\bf nodal sets}
$$
B_E^{-1}(0) := \lbrace x\in \R^2 : B_E(x) = 0\rbrace \,\, \mbox{ and } \,\,  (B^{\mathbb C}_E)^{-1}(0) =
B_E^{-1}(0)\cap (\widehat{B}_E)^{-1}(0)\,.
$$
It is proved in \cite[Lemma 8.3]{NPR:19} that $B_E^{-1}(0)$ is almost surely  a union of disjoint rectifiable curves (called {\bf nodal lines}), while $(B^{\mathbb C}_E)^{-1}(0)$ is almost surely  a locally finite collection of isolated points (often referred to as {\bf phase singularities} or {\bf optical vortices}, see e.g. \cite{DOP:16, UR:13}). 

\smallskip 

Now denote by $\mathscr{A} $ the collection of all piecewise $C^1$ simply connected compact subsets of $\R^2$ having non-empty interior, that is: $D\in \mathscr{A}$ if and only if $D$ is a simply connected compact set with non-empty interior, and with a piecewise $C^1$ boundary. A direct adaptation of \cite[Lemma 8.3]{NPR:19} (that only deals with convex bodies with $C^1$ boundary, but the generalisation is straightforward, since the only element used in the proof is the piecewise smoothness of boundaries) shows that, if $D\in \mathscr{A}$ is fixed, then almost surely  $B_E^{-1}(0)$ intersects $\partial D$ in at most a finite number of points, whereas the intersection $(B^{\mathbb C}_E)^{-1}(0)\cap \partial D$ is almost surely  empty. We will also denote by $\mathscr{A}_0 \subset \mathscr{A}$ the family of convex bodies of $\R^2$ having a $C^1$ boundary, that is: $D\in \mathscr{A}_0$ if and only if $D$ is a convex compact set, having non-empty interior and a $C^1$ boundary. For $D\in \mathscr{A}$, we set 
\begin{flalign}
&\L_{E}(D):=\text{length}(B_E^{-1}(0)\cap D)\,,\\
&\NNN_{E}(D):= \# \, \{ \(B_E^{\mathbb C}\)^{-1}(0)  \cap D  \} \,. \label{rv_2}
\end{flalign}
As shown in the next section, the main goal of the present paper is to study the weak convergence of the set-indexed random fields
\begin{equation}\label{processes}
\{\L_{E}(D):\, D\in \A\}\quad \text{and} \quad \{\NNN_{E}(D):\, D\in \A\}\,,
\end{equation}
in the sense of finite-dimensional distributions.

%%%%%%%%%%%%%%%%%%%%%%%%%%%%%%%%%%%%%%%%%%%%%%%%%%%%%%%%%%%%%%%%%%%%%%%%%%%%%%%%%%%%%%%%%%%%%%%%%%%%%%%%%%%%%%%%%%%%%

\section{Main results}\label{ss:intromain}

%%%%%%%%%%%%%%%%%%%%%%%%%%%%%%%------------------------------------------------------------------------------------------------------------------------------------------------------------------------------------------------------

\subsection{Multivariate CLTs}

The following statement contains fundamental results from \cite{Be:02} (the mean and variance computations in \eqref{meanlength}--\eqref{e:varlength}) and from \cite{NPR:19} (the one-dimensional CLTs stated in \eqref{e:1clt}). 

\begin{thm}[See \cite{Be:02, NPR:19}]\label{NPR-MR}
Let the above notation prevail and fix $D\in \mathscr{A}_0$. For $E>0$, the expectation of the nodal length $ \L_E(D)$ and of the number of phase singularities $\NNN_E(D)$ are, respectively, 
\begin{equation}\label{meanlength}
\E[ \L_E(D)] = \area(D)\,\frac{\pi}{\sqrt{2}}\sqrt{E} \quad\mbox{and} \quad  \E[\NNN_E(D)] = \area(D)\,\pi E\,,
\end{equation}
whereas the correspoinding variances verify the asymptotic relations
\begin{equation}\label{e:varlength}
\Var(\L_E(D)) \sim \area(D)\,\frac{1}{512\pi}\log E, \qquad \Var(\NNN_E(D)) \sim \area(D)\,\frac{11}{32\pi}E \log E, \quad E\to\infty\,. 
\end{equation}
Now let 
$$
\widetilde{\L}_{E}(D):=\frac{\L_{E}(D)-\E(\L_{E}(D))}{\sqrt{\Var(\L_{E}(D)})} \quad \text{and} \quad \widetilde{\NNN}_{E}(D):=\frac{\NNN_{E}(D)-\E(\NNN_{E}(D))}{\sqrt{\Var(\NNN_{E}(D)})}\,.
$$
Then, as $E\to \infty$, one has that
\begin{equation}\label{e:1clt}
\widetilde{\L}_{E}(D) \,,\,\,\widetilde{\NNN}_{E}(D) \Longrightarrow N\,,
\end{equation}
where $N\sim \NN(0,1)$ is a standard Gaussian random variable. 
\end{thm}

\begin{oss}{\rm We will see below that one of our technical findings (namely, the forthcoming Proposition \ref{prop1}), allows one to extend the content of Theorem \ref{NPR-MR} to the larger class $\mathscr{A}$.
}
\end{oss}

The key tool in the proof of Theorem \ref{NPR-MR} is an explicit computation of the {\bf Wiener-It\^o chaos expansion} of the two quantities $\widetilde{\L}_{E}(D)$ and $\widetilde{\NNN}_{E}(D)$ (see \cite[Chapter 2]{NP:12}, as well as Appendix A below). Such an approach reveals that, in the high-energy limit $E\sto\infty$, the fluctuations of $\widetilde{\L}_{E}(D)$ and $\widetilde{\NNN}_{E}(D)$ are completely determined by their projections on the fourth {\bf Wiener chaos} generated, respectively, by $B_E$ and $B_E^{\mathbb C}$. This observation provides a complete explanation of some striking {\bf cancellation phenomena} for nodal length variances observed by Berry \cite{Be:02}, and then confirmed in \cite{Wi:10} and \cite{KKW:13} for the models of random spherical harmonics and arithmetic random waves. The first paper connecting cancellation phenomena (for the variance of nodal lengths of random waves) to Wiener chaos expansions is \cite{MPRW:16}, dealing with the arithmetic case. Further studies in this direction for related models can be found in \cite{BM:17, RW:18, PR:18, To:18b, DNPR:19, MRW:19, BMW:19, CM:19, Ca:19}. We will see below that Wiener chaos expansions play an equally fundamental role in our findings.

\smallskip

Although Theorem \ref{NPR-MR} applies to generic elements of $\A_0$, it does not provide any information about the asymptotic dependence structure of random vectors of the type $(\widetilde{\L}_{E}(D_1),\dots,\widetilde{\L}_{E}(D_m))$ or $(\widetilde{\NNN}_{E}(D_1),\dots,\widetilde{\NNN}_{E}(D_m))$.  The next statement fills such a gap by providing a non-trivial multivariate extension of Theorem \ref{NPR-MR}; it is the main result of the paper.

\medskip

\begin{thm}[\bf Multivariate CLT for nodal lengths and phase singularities] \label{thm_L}
For $m\geq 1$, fix $D_1,D_2,\dots,D_m\in\A$, and define the $m\times m$ matrix $C = \{C_{i,j} \}$ by the relation
\begin{equation}\label{lim_cov_ij}
C_{i,j}:=\frac{\area\(D_i\cap D_j\)}{\sqrt{\area\(D_i\)\, \area\(D_j\)}}\,.
\end{equation}
Then, as $E\to \infty$, one has that
\begin{equation}\label{e:mainlength}
\(\widetilde{\L}_{E}(D_1),\widetilde{\L}_{E}(D_2),\dots,\widetilde{\L}_{E}(D_m)\) {\implies}N (0,C)\,,
\end{equation}
and  
\begin{equation}\label{e:mainnumber}
\(\widetilde{\NNN}_{E}(D_1),\widetilde{\NNN}_{E}(D_2),\dots,\widetilde{\NNN}_{E}(D_m)\){\implies}N(0,C)\,,
\end{equation}
where $N(0,C)$ indicates an $m$-dimensional centered Gaussian vector with covariance $C$.
\end{thm}

Theorem \ref{thm_L} implies in particular that, if $D_1\cap D_2 = \emptyset$, then the two random variables $\widetilde{\L}_{E}(D_1)$ and $\widetilde{\L}_{E}(D_2)$ (resp. $\widetilde{\NNN}_{E}(D_1)$ and $\widetilde{\NNN}_{E}(D_2)$) are asymptotically independent. Relations \eqref{e:mainlength} and \eqref{e:mainnumber} also contain a generalisation of \eqref{e:varlength}, that we present in the next statement.

\begin{cor} For every $D_1, D_2 \in \A$,
$$
\frac{ \Cov(\L_E(D_1), \L_E(D_2))}{ (\log E)/ (512\pi)},\quad \frac{ \Cov(\NNN_E(D_1), \NNN_E(D_2))}{(11 E \log E) /(32\pi)} \longrightarrow\area(D_1\cap D_2),
$$
as $E\to\infty$.
\end{cor}

\begin{oss}\label{r:meas}{\rm Let $\mathscr{B}_0$ denote the class of Borel subsets of $\R^2$ having finite Lebesgue measure, and observe that $\A\subset \mathscr{B}_0$. Following e.g. \cite[Chapter 2]{NP:12}, we define a homogeneous {\bf independently scattered Gaussian random measure} on $\R^2$, to be a centered Gaussian family
$$
{\bf G}  = \{G(C) : C\in \mathscr{B}_0\},
$$
verifying the following relation: for every $C_1,C_2\in \mathscr{B}_0$, $\E[G(C_1)G(C_2)] = \area(C_1\cap C_2)$ (a self-contained proof of the existence of such an object can be found in \cite[p.24]{NP:12}). In view of such a definition, the content of Theorem \ref{thm_L} can be reformulated in the following way: as $E\sto\infty$, the two set-indexed processes
$$
\sqrt{\frac{512\pi}{ \log E}}\{\L_{E}(D) - \E(\L_E(D)) :\, D\in \A\}\quad 
$$ 
and
$$  
\sqrt{\frac{32\pi}{11\, E \log E}} \{\NNN_{E}(D) - \E(\NNN_E(D)) :\, D\in \A\}\,
$$
converge to the restriction of ${\bf G}$ to $\mathscr{A}$ in the sense of finite-dimensional distributions. Whether such a convergence takes in place in a stronger functional sense (see e.g. \cite{BR:85}) is an open problem, whose complete solution seems to be still outside the scope of existing techniques. The next section contains some further discussion in this direction.}
\end{oss}

%%%%%%%%%%%%%%%%%%%%%%%%%%%%%%%------------------------------------------------------------------------------------------------------------------------------------------------------------------------------------------------------

\subsection{A (partial) weak convergence result}

We recall that a {\bf Wiener sheet} on $[0,1]^2$ is a centered Gaussian field $$ {\bf W} = \{W(t_1,t_2) : t_1, t_2\geq0\},$$ such that $\E[W(t_1,t_2)W(s_1,s_2)] = (t_1\wedge s_1 )(t_2\wedge s_2)$ and the mapping $(t_1,t_2) \mapsto W(t_1,t_2)$ is almost surely  continuous (see e.g. \cite[p. 39]{RY:99} for an introduction to such an object). Now consider the two centered random fields on $[0,1]^2$ given by
$$
X_E(t_1,t_2) := \sqrt{\frac{512\,\pi}{\log E}}\,\Big(\L_E([0,t_1]\times [0,t_2]) - \E(\L_E([0,t_1]\times [0,t_2]))\Big)
$$
and
$$
Y_E(t_1,t_2) := \sqrt{\frac{32\,\pi}{E \log E}}\,\Big(\NNN_E([0,t_1]\times [0,t_2]) - \E(\NNN_E([0,t_1]\times [0,t_2]) )\Big).
$$
Both $X_E$ and $Y_E$ belong almost surely  to the Skorohod space ${\bf D}_2$ of `cadlag' functions on $[0,1]^2$, as defined e.g. in the classical reference \cite{Ne:71}. One immediate consequence of Theorem \ref{thm_L} (and Remark \ref{r:meas}) is that, as $E\sto \infty$, both $X_E$ and $Y_E$ converge to ${\bf W}$ in the sense of finite-dimensional distributions, and a natural question is whether such a convergence can be lifted to {\bf weak convergence} in the metric space ${\bf D}_2$ (see again \cite{Ne:71}). Proving such a functional result would typically allow one to deduce a number of novel limit theorems (involving e.g. the global and local maxima and minima of $X_E$ and $Y_E$),  as a consequence of the well-known {Continuous Mapping Theorem} (see e.g. \cite{B:99}).  Similarly to what is observed at the end of Remark \ref{r:meas}, a complete solution to this problem seems to require novel ideas. A first step in this direction is contained in the next statement. From now on, we will denote by $C([0,1]^2)$ the space of continuous real-valued functions on $[0,1]^2$, that we endow with the metric induced by the supremum norm.

\begin{thm}\label{tightness}
For every $t_1,t_2\in [0,1]^2$ and every $E>0$, define $X_E^{[4]}(t_1,t_2)$ (resp. $Y_E^{[4]}(t_1,t_2)$) to be the projection of $X_E(t_1,t_2)$ (resp. $Y_E(t_1,t_2)$) onto the fourth Wiener chaos generated by $B_E$ (resp. $B_E^{\mathbb C}$). Then, for every fixed $(t_1,t_2)$, $\E[(X_E^{[4]}(t_1,t_2) - X_E(t_1,t_2))^2]\sto 0$ and $\E[(Y_E^{[4]}(t_1,t_2) - Y_E(t_1,t_2))^2]\sto 0$, as $E\sto \infty$.  Moreover, the random mappings $(t_1,t_2)\mapsto X_E^{[4]}(t_1,t_2)$ and $(t_1,t_2)\mapsto X_E^{[4]}(t_1,t_2)$ belong almost surely  to the space $C([0,1]^2)$ and, as $E\sto\infty$, both $X_E^{[4]}$ and $Y_E^{[4]}$ converge weakly to ${\bf W}$, that is: for every continuous bounded mapping $\varphi : C([0,1]^2) \sto \R$,
$$
\E\big[\varphi\big(X_E^{[4]}\big)\big], \, \E\big[\varphi\big (Y_E^{[4]}\big )\big ] \longrightarrow \E\big[\varphi\big({\bf W}\big)\big]. 
$$
\end{thm}

Theorem \ref{tightness} is proved in Section \ref{ss:prooftightness}.

%%%%%%%%%%%%%%%%%%%%%%%%%%%%%%%%%%%%%%%%%%%%%%%%%%%%%%%%%%%%%%%%%%%%%%%%%%%%%%%%%%%%%%%%%%%%%%%%%%%%%%%%%%%%%%%%%%%%

\section{Application to monochromatic and superposed waves}\label{ss:introapp}

We will now show that the main results of our paper allow one to deduce multivariate CLTs for (a) pullback random waves defined on general 2-dimensional manifolds, and (2) non-Gaussian waves obtained as the superposition of independent trigonometric waves with random directions and phases.

%%%%%%%%%%%%%%%%%%%%%%%%%%%%%%%------------------------------------------------------------------------------------------------------------------------------------------------------------------------------------------------------

\subsection{Monochromatic waves}

Let $(\mathcal M, g)$ be a 2-dimensional compact smooth Riemannian manifold. We denote by $\Delta_g$ the Laplace-Beltrami operator on $\mathcal M$, and write $\lbrace f_j : j\in \N \rbrace$ to indicate an orthonormal basis of $L^2(\mathcal M)$, composed of real eigenfunctions of $\Delta_g$ such that
$$
\Delta_g f_j + \lambda^2_j f_j=0,
$$
where the eigenvalues are implicitly ordered in such a way that $0=\lambda_0<\lambda_1\le \lambda_2 \le \dots \uparrow \infty$. Following e.g. \cite{CH:16, Ze:09}, the {\bf monochromatic random wave} on $\mathcal M$ of parameter $\lambda$ is defined as the  random field
\begin{equation}\label{phi}
\phi_\lambda(x) := \frac{1}{\sqrt{\text{dim}(H_{c,\lambda})}} \sum_{\lambda_j\in [\lambda, \lambda + 1]} a_j f_j(x), \quad x\in M,
\end{equation}
where the $a_j$ are i.i.d. standard Gaussian and
$$
H_{\lambda} := \bigoplus_{\lambda_j \in [\lambda, \lambda + 1]} \text{Ker}(\Delta_g + \lambda_j^2\, \text{Id}),
$$
with $\text{Id}$ the identity operator. The field $\phi_\lambda$ is of course centered and Gaussian, and its covariance kernel is given by 
\begin{equation}\label{covRiem}
K_{\lambda}(x,y) := \Cov(\phi_\lambda(x), \phi_\lambda(y)) = \frac{1}{\text{dim}(H_{\lambda})} \sum_{\lambda_j\in [\lambda, \lambda +1]} f_j(x) f_j(y),\quad x,y\in \mathcal M.
\end{equation}
``Short window'' monochromatic random waves such as $\phi_\lambda$ (for manifolds of any dimension) were first introduced by Zelditch in \cite{Ze:09} as approximate models of random Laplace eigenfunctions on manifolds that do not necessarily possess spectral multiplicities; see \cite{CH:16, SW:16, NS:16, BW:18, CS:19, NPR:19} for further references and details. 

\smallskip

Following \cite{CH:16}, we now fix $x\in \mathcal M$, and consider the tangent plane $T_x\mathcal M \simeq \R^2$ to the manifold at $x$. We define the {\bf pullback random wave} generated by $\phi_\lambda$ at $x$ to be the Gaussian random field on $T_x\mathcal M$ given by 
$$
\phi_\lambda^x(u) := \phi_\lambda\left ( \exp_x \left ( \frac{u}{\lambda} \right )\right ),\qquad u\in T_x \mathcal M, 
$$
where $\exp_x : T_x\mathcal M \to \mathcal M$ is the exponential map at $x$. The planar field $\phi_\lambda^x$ is of course centered, and Gaussian and its covariance kernel is
$$
K_{\lambda}^x(u,v) = K_{\lambda}\left(\exp_x \left ( \frac{u}{\lambda}  \right) , \exp_x \left ( \frac{v}{\lambda}\right ) \right ),\qquad u,v\in T_x \mathcal M.  
$$ 

\begin{defn}[See \cite{CH:16}]{\rm We say that $x \in \mathcal{M}$ is a point of {\bf isotropic scaling} if, for every positive function $\lambda \mapsto r(\lambda)$ such that $r(\lambda) = o(\lambda)$, one has that
\begin{equation}\label{limit}
\sup_{u,v \in \mathbb{B}(r(\lambda))  } \left| \partial^\alpha\partial^\beta[  K_{\lambda}^x(u,v) - (2\pi) J_0(\| u-v\|_{g_x})]\right| \to 0,\quad \lambda\to \infty,
\end{equation}
where $\alpha, \beta\in \mathbb{N}^2$ are multi-indices classifying partial derivatives with respect to $u$ and $v$, respectively, $\| \cdot \|_{g_x}$ is the norm on $T_x\mathcal M$ induced by $g$, and $\mathbb{B}(r(\lambda))$ is the ball of radius $r(\lambda)$ containing the origin.
}
\end{defn}

Sufficient conditions for a point $x$ to be of isotropic scaling are presented e.g. in \cite[Section 2.5]{CH:16}, and the references therein. We observe that it is always possible to choose coordinates around $x$ in such a way that $g_x = \text{Id}$, and in this case the limiting kernel in \paref{limit} coincides with the covariance of the Gaussian field $\sqrt{2\pi}\cdot  b_1$, as defined in Remark \ref{r:not}. It follows that, if $x$ is a point of isotropic scaling and $g_x$ has been chosen as above, then, as $\lambda \sto \infty$, the planar field $\phi_\lambda^x$ converges $\sqrt{2\pi}\cdot b_1$, in the sense of finite-dimensional distributions. 

\smallskip

Keeping the above notation and assumptions, we now state a special case of \cite[Theorem 1]{CH:16}. For this, we will need the following notation: for $x\in \mathcal{M}$, and $D\subset T_x \mathcal M$,
$$
\mathcal{Z}_{\lambda, E}^x(D) := {\rm length} \big\lbrace (\phi_\lambda^x)^{-1}(0)\cap 2\pi\sqrt{E} \cdot D \big\rbrace, \quad E>0,
$$
where, for $c>0$, $c\cdot D:= \{y : cx, \mbox{ for some } x\in D\}$. The next statement shows that, if $x$ is of isotropic scaling, then, for $\lambda$ sufficiently large, vectors of random variables of the type $\mathcal{Z}_{\lambda, E}^x(D)$ behave like the corresponding vectors of nodal lengths for Berry's random waves.

\begin{thm}[Special case of Theorem 1 in \cite{CH:16}]\label{t:ch}Let $x$ be a point  of isotropic scaling, and assume that coordinates have been chosen around $x$ in such a way that $g_x = {\rm Id.}$ Fix $E>0$, as well as balls closed balls $B_1,...,B_m $. Then, as $\lambda \sto \infty$,  the random vector $(\mathcal{Z}_{\lambda, E}^x(B_1),...,\mathcal{Z}_{\lambda, E}^x(B_m))$ converges in distribution to 
\begin{eqnarray*}
&& \left( {\rm length}\Big(b_1^{-1}(0)\cap  2\pi\sqrt{E}\cdot B_1 \Big),...., {\rm length}\Big(b_1^{-1}(0)\cap  2\pi\sqrt{E}\cdot B_m \Big) \right)\\
&& \left ( \stackrel{d}{=} 2\pi\sqrt{E} \cdot \left(\mathscr{L}_E(B_1),\dots, \mathscr{L}_E(B_m) \right) \right ),
\end{eqnarray*}
where the identity in distribution stated between brackets follows from the fact that, as random fields, $B_E(x)$ and $b_1(2\pi\sqrt{E} x)$ have the same law.

\end{thm}

The next statement (whose simple proof -- analogous to the one of \cite[Theorem 1.8]{NPR:19} -- is omitted for the sake of brevity) is a direct consequence of Theorem \ref{thm_L} and Theorem \ref{t:ch}, and provides both a second-order counterpart to Theorem \ref{t:ch} and a multivariate extension of \cite[Theorem 1.8]{NPR:19}. This shows in particular that nodal lengths of pullback random waves display multivariate high-energy Gaussian fluctuations reproducing the ones of Berry's model, at every point of isotropic scaling. We use the shorthand notation:
$$
\widetilde{\mathcal{Z}}_{\lambda, E}^x(D) :=  \frac{\mathcal{Z}_{\lambda, E}^x(D)}{2\pi\sqrt{E}}. 
$$

\begin{thm}[CLT for the nodal length of pullback waves]\label{t:rrw} Let $x$ be a point  of isotropic scaling, and assume that coordinates have been chosen around $x$ in such a way that $g_x = {\rm Id.}$ Fix closed balls $
B_1,...,B_m $, and let $\{E_k : k\geq 1\}$ be a sequence of positive numbers such that $E_k \to \infty$. Then, there exists a sequence $\{\lambda_k : k\geq 1\}$ such that, as $k\to \infty$, the vector
$$
\( \frac{\widetilde{\mathcal{Z}}_{\lambda_k, E_k}^x(B_1) - \area(B_1) \pi\sqrt{E_k/2}}{\sqrt{\area( B_1 )\log E_k/ (512\pi)   }}, ..., \frac{\widetilde{\mathcal{Z}}_{\lambda_k, E_k}^x(B_m) - \area(B_m) \pi\sqrt{E_k/2}}{\sqrt{\area( B_m )\log E_k/ (512\pi)   }} \)
$$
converges in distribution to a centered $m$-dimensional Gaussian vector with the same covariance matrix $C$ defined in Theorem \ref{thm_L} for $B_i = D_i$, $i=1,...,m$.
\end{thm}

As for \cite[Theorem 1.8]{NPR:19},  a shortcoming of the previous statement is that it does not provide any quantitative information about the sequence $\{\lambda_k\}$. As already observed in \cite[Section 1.4.3]{NPR:19}, in order to obtain a more precise statement, one would need some explicit estimate on the speed of convergence to zero of the supremum appearing in \eqref{limit}. Obtaining such estimates is a rather challenging problem; see \cite{K:19} for some recent advances.

%%%%%%%%%%%%%%%%%%%%%%%%%%%%%%%------------------------------------------------------------------------------------------------------------------------------------------------------------------------------------------------------

\subsection{Superposition of trigonometric random waves}

In the already discussed paper \cite{Be:02}, Berry proposed a simple random model for the statistics of nodal lines of Laplace eigenfunctions defined on chaotic quantum billiards. In particular, in \cite{Be:02} it is conjectured that the zero set of deterministic wavefunctions with wavenumber $k$, for highly excited chaotic states $k\gg 1$, behaves locally as the one of a {\bf superposition of independent random wavefunctions}, having all the same wavenumber $k$, but different directions. Formally, such a superposition is defined as
\begin{equation}\label{rw}
u_{J;k}(x):=\sqrt{\frac{2}{J}}\sum_{j=1}^J\cos\(kx_1\cos\theta_j+kx_2\sin\theta_j+\phi_j\), \quad x= (x_1,x_2)\in\R^2, \quad J\gg1,
\end{equation}
where $\theta_j$ and $\phi_j$ are, respectively, random directions and random phases such that $\(\theta_1,\phi_1,\dots,\theta_J,\phi_J\)$ are i.i.d. uniform random variables on $[0,2\pi]$; observe that $\E[u_{J;k}(x) ]= 0$ for every $x$. Our aim in this section is to illustrate how one can take advantage of the main findings of the present paper, in order to characterise the local high-energy fluctuations of the nodal lengths of the field $u_{J;k}$, when restricted to the square $[0,1]^2$ (the discussion below actually applies to any bounded subset of $\R^2$ -- the choice of $[0,1]^2$ being only motivated by notational convenience).

\smallskip

For $i=1,2$, we set $\partial_i := \partial/\partial x_i$, and define $\partial_0$ to be the identity operator. In what follows, we denote by $C^1([0,1]^2)$ the class of continuous real-valued mappings on $[0,1]^2$ having continuous first order partial derivatives. Recall that $C^1([0,1]^2)$ is a Polish space, when endowed with the metric induced by the norm $\| f\| = \sum_{i=0}^2 \| \partial_i f\|_\infty$.

\smallskip

An application of the classical multivariate CLT, together with some standard covariance computations, reveals that, for every $d\geq 1$, for every $(i_1,...,i_d)\in \{0,1,2\}^d$ and every $x^1,...,x^d\in [0,1]^2$,
$$
(\partial_{i_1} u_{J;k}(x^1), ..., \partial_{i_d} u_{J;k}(x^d))\implies (\partial_{i_1} b_k (x^1), ..., \partial_{i_d} b_k(x^d)),\quad J\to\infty,
$$
that is: the 3-dimensional field $(u_{J;k},\partial_1 u_{J;k}, \partial_2 u_{J;k})$ converges to $(b_k, \partial_1 b_k, \partial_2 b_k)$, in the sense of finite-dimensional distributions. Tedious but standard computations (omitted) also show that, for every fixed $k$, there exists a finite constant $B = B(k)$ such that, for every $i=0,1,2$, and every $x,y\in [0,1]^2$
$$
\E[ ( u_{J;k}(x) - u_{J;k}(y))^6]\leq B \|x-y\|^3.  
$$
We can now apply \cite[Theorem 2 and Remark 2]{RS:01}, together with \cite[Theorem 3]{APP:17} and the Continous Mapping Theorem \cite[Theorem 2.7]{B:99} to deduce that, as $J\sto \infty$: (a) $u_{J;k}$ weakly converges to $b_k$ in the space $C^1([0,1]^2)$ (that is, $\E[\varphi(u_{J;k})]\sto \E[\varphi(b_k)]$ for every $\varphi : C^1([0,1]^2)\sto \R$ continuous and bounded), and (b) for every collection $D_1,...,D_m$ of compact subsets of $[0,1]^2$,
\begin{eqnarray*}
U(J,k,m) &:=& \left({\rm length}( (u_{J;k})^{-1}(0)\cap D_1),....,{\rm length} ((u_{J;k})^{-1}(0)\cap D_m)\right) \\ 
&\implies& \left({\rm length}( (b_k)^{-1}(0)\cap D_1),....,{\rm length} (b_k)^{-1}(0)\cap D_m)\right).
\end{eqnarray*}
Now denote by $\widehat{U}(J,k,m)$ the normalised version of the vector $U(J,k,m)$ defined above, obtained by replacing each random variable ${\rm length}(u_{J;k})^{-1}(0)\cap D_i)$, $i=1,...,m$, by the quantity
$$
\frac{{\rm length}( (u_{J;k})^{-1}(0)\cap D_i) - \area(D_i)k/\sqrt{8} }{\sqrt{\area(D_i)\log k/256\pi}}
$$
(observe that, according to Theorem \ref{NPR-MR}, $\E[ {\rm length}( (b_k)^{-1}(0)\cap D_1)] = \area(D_i)k/\sqrt{8}$ and, as $k\to\infty$, $\Var [ {\rm length}( (b_k)^{-1}(0)\cap D_i)] \sim \area(D_i)\log k/256\pi$). 

\smallskip

Reasoning as in the proof of \cite[Theorem 1.8]{NPR:19}, we can therefore deduce the following consequence of Theorem \ref{thm_L}.

\begin{thm}\label{t:srw} Let $\{J_n\}$ be a sequence of integers diverging to infinity. Then, there exists a sequence $\{k_n\}$ such that $k_n\sto \infty$, and
$$
\widehat{U}(J_n,k_n,m)\implies N(0,C), \quad n\to\infty,
$$
where $N(0,C)$ denotes a $m$-dimensional centered Gaussian vector with covariance $C$ as in Theorem \ref{thm_L}.

\end{thm}

As for Theorem \ref{t:rrw}, the statement of Theorem \ref{t:srw} does not provide any quantitative information about the sequence $\{k_n\}$. In order to deduce a more informative conclusion, one would in principle need to explicitly couple the two fields $u_{J;k}$ and $b_1$ on the same probability space, and then to use such a coupling in order to assess the distance between the distribution of $U(J,k,m)$ and that of $ \left({\rm length}( (b_k)^{-1}(0)\cap D_1),....,{\rm length} (b_k)^{-1}(0)\cap D_m)\right)$. We prefer to regard this demanding task as a separate problem, and leave it open for further investigation.

\medskip

The rest of the paper is devoted to the proof of our main findings.

%%%%%%%%%%%%%%%%%%%%%%%%%%%%%%%%%%%%%%%%%%%%%%%%%%%%%%%%%%%%%%%%%%%%%%%%%%%%%%%%%%%%%%%%%%%%%%%%%%%%%%%%%%%%%%%%%%%%

\section{Proofs of main results}\label{proofs}

%%%%%%%%%%%%%%%%%%%%%%%%%-----------------------------------------------------------------------------------------------------------------------------------------------------------------------------------------------------------------------------------------------------

\subsection{Proof of Theorem \ref{thm_L} }

In order to prove our main results, we first need to establish two technical statements, substantially extending \cite[Propositions 5.1 and 5.2]{NPR:19}. 
\medskip

From now on, for any $D\in \A$, we set $\diam (D):=\sup_{x,y\in D}\norm{x-y}$ (with $\diam \emptyset = 0$ by definition)
and define, for each $\eta \geq 0$,
\begin{flalign*}
&D^{+\eta}:=\{x\in\R^2: \operatorname{dist}(x,D)\leq\eta\} \quad\mbox{and}\quad  D^{-\eta}:=\{x\in D: \operatorname{dist}(x,\partial D)\geq\eta\}\,.
\end{flalign*}

{{\begin{prop}\label{prop1}
Let $q_{i,j}\geq 1$ for $i,j=0,1,2$ and $\sum_{i,j=0}^2\, q_{i,j}=4$. Then, for all $D_1,D_2 \in \A$ one has that, as $E\sto\infty$,
\begin{flalign}
&\int_{D_1}\int_{D_2}\,\prod_{i,j=0}^2\,\widetilde{r}^E_{i,j}(x-y)^{q_{i,j}}\,dx\,dy=\label{eq_prop1}\\
&=\int_{0}^{\diam(D_1 \cap D_2)}\,d\phi \,\area\(D_1\cap D_2^{-\phi}\)\,\int_{0}^{2\pi}\,\prod_{i,j=0}^2\,\tilde{r}^E_{i,j}(\phi\cos\theta,\phi\sin\theta)^{q_{i,j}}\,\phi\,d\theta+o\(\frac{\log E}{E}\)\,.\notag
\end{flalign}
\end{prop}

\begin{oss}\label{empty}
Relation \eqref{eq_prop1} yields in particular that, if $D_1,D_2 \in \A$ are such that $\area\(D_1\cap D_2\)=0$, then, as  $E\sto\infty$, 
\begin{flalign}\label{eq2_prop1}
\int_{D_1}\int_{D_2}\,\prod_{i,j=0}^2\,\widetilde{r}^E_{i,j}(x-y)^{q_{i,j}}\,dx\,dy=o\(\frac{\log E}{E}\)\,.
\end{flalign}
\end{oss}
}}
\begin{oss}\label{r:papeete} It is shown in \cite[Proposition 5.1]{NPR:19} that, if $D_1=D_2\in \mathscr{A}_0$, then, in \eqref{eq_prop1} and \eqref{eq2_prop1} one can replace the symbol $o(\log E /E)$ with $O(1/E)$, which provides of course a stronger estimate. By inspection of the arguments developed in \cite{NPR:19}, one also observes that: (i) the estimate $o(\log E /E)$ is the only one needed in the proofs of Theorem \ref{NPR-MR}, and (ii) the proof of \cite[Proposition 5.1]{NPR:19} is the only place in \cite{NPR:19} where convexity is used (since the argument used therein exploits Steiner's formula for convex sets). It follows in particular that, thanks to our Proposition \ref{prop1}, the conclusion of Theorem \ref{NPR-MR} can be extended to the larger class $\mathscr{A}$.

\end{oss}

\begin{proof}[Proof of Proposition \ref{prop1}]
Without loss of generality, we can assume that $E>1$. 
Using the coarea formula, we deduce that
\begin{flalign}
&\,\int_{D_1}\int_{D_2}\,\prod_{i,j=0}^2\,\widetilde{r}^E_{i,j}(x-y)^{q_{i,j}}\,dx\,dy\notag\\
&=\,\int_{0}^{\diam(D_1 \cup D_2)}\,d\phi \,\int_{D_1} \,dx\int_{\partial B_{\phi}(x)\cap D_2}\,\prod_{i,j=0}^2\,\widetilde{r}^E_{i,j}(x-y)^{q_{i,j}}\,dy\notag\\
&=\,\int_{0}^{\diam(D_1 \cup D_2)}\,d\phi \,\int_{D_1}\,\1_{\{x\in D_1:\partial B_{\phi}(x)\cap D_2\neq\varnothing\}} \,dx\int_{\partial B_{\phi}(x)\cap D_2}\,\prod_{i,j=0}^2\,\widetilde{r}^E_{i,j}(x-y)^{q_{i,j}}\,dy\notag\\
&=\,\int_{0}^{\diam(D_1 \cap D_2)}\,d\phi \,\int_{D_1}\,\1_{\{x\in D_1:\partial B_{\phi}(x)\cap D_2\neq\varnothing\}} \,dx\int_{\partial B_{\phi}(x)\cap D_2}\,\prod_{i,j=0}^2\,\widetilde{r}^E_{i,j}(x-y)^{q_{i,j}}\,dy\notag\\
&\qquad+\,\int_{\diam(D_1 \cap D_2)}^{\diam(D_1 \cup D_2)}\,d\phi \,\int_{D_1}\,\1_{\{x\in D_1:\partial B_{\phi}(x)\cap D_2\neq\varnothing\}} \,dx\int_{\partial B_{\phi}(x)\cap D_2}\,\prod_{i,j=0}^2\,\widetilde{r}^E_{i,j}(x-y)^{q_{i,j}}\,dy\notag\\
&=\,\int_{0}^{\diam(D_1 \cap D_2)}\, \,d\phi \,\int_{D_1\cap D_2^{-\phi}}\,dx\,\int_{0}^{2\pi}\,\prod_{i,j=0}^2\,\tilde{r}^E_{i,j}(\phi\cos\theta,\phi\sin\theta)^{q_{i,j}}\,\phi\,d\theta
\notag\\
&\qquad+\underbracket[0.5pt][5pt]{\,\int_{0}^{\diam(D_1 \cap D_2)}\,d\phi \,\int_{D_1\cap \(D_2^{+\phi}\setminus D_2^{-\phi}\)}\,dx\int_{\partial B_{\phi}(x)\cap D_2}\,\prod_{i,j=0}^2\,\tilde{r}^E_{i,j}(x-y)^{q_{i,j}}\,dy}_{=: \, C_2}\label{C_2}\\
&\qquad+\underbracket[0.5pt][5pt]{\,\int_{\diam(D_1 \cap D_2)}^{\diam(D_1 \cup D_2)}\,d\phi \,\int_{D_1\cap D_2^{+\phi}}\,dx\int_{\partial B_{\phi}(x)\cap D_2}\,\prod_{i,j=0}^2\,\tilde{r}^E_{i,j}(x-y)^{q_{i,j}}\,dy}_{=:\, C_3} \label{C_3}\\
&=\int_{0}^{\diam(D_1 \cap D_2)}\,d\phi\,\area\(D_1\cap D_2^{-\phi}\)\,\int_{0}^{2\pi}\,\phi\,d\theta\,\prod_{i,j=0}^2\,\tilde{r}^E_{i,j}(\phi\cos\theta,\phi\sin\theta)^{q_{i,j}}+C_2+C_3\,,\notag
\end{flalign}
where the symbol $\1$ stands for the indicator function.
We notice the following special cases: (i) $\diam(D_1\cap D_2) = \diam(D_1\cup D_2)$, which implies $C_3 = 0$, and (ii) $\diam(D_1\cap D_2) =0$, in which case $C_3$ is the only non-zero term of the previous sum. To deal with $C_2$, we first pass to polar coordinates $y_1=x_1+\phi\cos\theta$, $y_2=x_2+\phi\sin\theta$, and then we perform the change of variable $\phi=\psi/\sqrt E$, to have 
\begin{flalign}
&\abs{C_2}%&=\int_{0}^{\diam(D_1 \cap D_2)}\,d\phi \,\int_{D_1\cap \(D_2^{+\phi}\setminus D_2^{-\phi}\)}\,dx\int_{\partial B_{\phi}(x)\cap D_2}\,J_0\(2\pi\sqrt{E}\norm{x-y}\)^4\,dy\\
=\abs{\int_{0}^{\diam(D_1 \cap D_2)}\,\phi\,d\phi \,\int_{D_1\cap \(D_2^{+\phi}\setminus D_2^{-\phi}\)}\,dx\int_{\partial B_{\phi}(x)\cap D_2}\,\prod_{i,j=0}^2\,\tilde{r}^E_{i,j}(\phi\cos\theta,\phi\sin\theta)^{q_{i,j}}}\notag\\
&\leq\int_{0}^{\sqrt{E}\diam(D_1 \cap D_2)}\,\frac{d\psi}{E} \,\int_{D_1\cap \(D_2^{+\psi/\sqrt{E}}\setminus D_2^{-\psi/\sqrt{E}}\)}\,dx\,\int_0^{2\pi}\,\prod_{i,j=0}^2 \abs{\tilde{r}^1_{i,j}(\psi\cos\theta,\psi\sin\theta)^{q_{i,j}}}\, \psi\,d\theta\notag.
\end{flalign}
We now split the integral on the right-hand side of the previous inequality as $$\int_{0}^{\sqrt{E}\diam(D_1 \cap D_2)} = \int_{0}^1+ \int_{1}^{\sqrt{E}\diam(D_1 \cap D_2)}\, ,$$ and denote the two resulting integrals as $C_{2,1}$ and $C_{2,2}$, respectively. Since $C_{2,1}$ is an integral over a fixed compact interval, we can directly use the fact that the kernels $\tilde{r}^1_{i,j}$ are all bounded by 1, to obtain that 
$$
|C_{2,1}|\leq \frac{2\pi}{E}\,\int_{0}^{1}\,\psi\,\area\(D_1\cap \(D_2^{+\psi/\sqrt{E}}\setminus D_2^{-\psi/\sqrt{E}}\)\)\,d\psi \label{psi}= O\left(\frac{1}{E}\right).
$$
To deal with $C_{2,2}$, we start by using the asymptotic relation \eqref{asymp_zero} in order to deduce that, for some absolute constant $c$,
\begin{equation}\label{c2.1}
|C_{2,2}| \leq \frac{c}{E}\,\int_{1}^{\sqrt{E}\,\diam(D_1 \cap D_2)}\,\frac{1}{\psi}\area\(D_1\cap \(D_2^{+\psi/\sqrt{E}}\setminus D_2^{-\psi/\sqrt{E}}\)\) \,d\psi.
\end{equation}
\noindent Now, as $E\sto \infty$, and since $D_2$ has a $C^1$ boundary, one has that, as $\alpha\downarrow 0$,
$$
\area\(D_1\cap \(D_2^{+\alpha}\setminus D_2^{-\alpha }\)\)\to 0 \,,
$$
which in turn implies that, $\forall\,\eps>0$, there exists $\delta >0$ such that 
$$
\area\(D_1\cap \(D_2^{+\psi/\sqrt{E}}\setminus D_2^{-\psi/\sqrt{E}}\)\) \leq \eps\,, \quad\text{whenever $\frac{\psi}{\sqrt{E}}<\delta$\,. }
$$
For a fixed $\eps>0$, we consequently select such a $\delta$ and write
\begin{flalign*}
&\int_{1}^{\sqrt{E}\,\diam(D_1 \cap D_2)}\,\frac{1}{\psi}\,\area\(D_1\cap  \(D_2^{+\psi/\sqrt{E}}\setminus D_2^{-\psi/\sqrt{E}}\)\)\,d\psi=\\
&=\underbracket[0.5pt][5pt]{\int_{1}^{\delta\sqrt{E}}\,\frac{1}{\psi}\,\area\(D_1\cap  \(D_2^{+\psi/\sqrt{E}}\setminus D_2^{-\psi/\sqrt{E}}\)\)\,d\psi}_{=:\,  C_{2,2,1}}\\
&\quad+\underbracket[0.5pt][5pt]{\int_{\delta\sqrt{E}}^{\sqrt{E}\,\diam(D_1 \cap D_2)}\,\frac{1}{\psi}\,\area\(D_1\cap  \(D_2^{+\psi/\sqrt{E}}\setminus D_2^{-\psi/\sqrt{E}}\)\)\,d\psi}_{=: \, C_{2,2,2}}\,;
\end{flalign*}
hence\footnote{Note that, if $\diam(D_1 \cap D_2)\le\delta$, there is no need of splitting the integral in the sum of $C_{2,2,1}$ and $C_{2,2,2}$, as in this case $$\int_{1}^{\sqrt{E}\diam(D_1 \cap D_2)}\,\frac{d\psi}{\psi}\area\(D_1\cap  \(D_2^{+\frac{\psi}{\sqrt E}}\setminus D_2^{-\frac{\psi}{\sqrt E}}\)\)\le \int_{1}^{\delta\sqrt{E}}\frac{d\psi}{\psi}\area\(D_1\cap  \(D_2^{+\frac{\psi}{\sqrt E} }\setminus D_2^{-\frac{\psi}{\sqrt E}}\)\)\,,$$ and the last integral equals $C_{2,2,1}$.}
\begin{flalign*}
C_{2,2,2}&\leq \area\(D_1\)\, \log\(\frac{\diam(D_1 \cap D_2)}{\delta}\) \quad\text{while}\quad 
C_{2,2,1}\leq \eps\, \log\(\delta\sqrt{E}\)\,.
\end{flalign*}
As a consequence
\begin{flalign*}
&\limsup_{E\to\infty}\frac{1}{\log\sqrt{E}}\int_{1}^{\sqrt{E}\,\diam(D_1 \cap D_2)}\,\frac{1}{\psi}\,\area\(D_1\cap \(D_2^{+\psi/\sqrt{E}}\setminus D_2^{-\psi/\sqrt{E}}\)\)\,d\psi\\
&\leq\limsup_{E\to\infty}\frac{1}{\log\sqrt{E}}\(\eps\, \log\(\delta\sqrt{E}\)+\area\(D_1\)\, \log\(\frac{\diam(D_1 \cap D_2)}{\delta}\)\)=\eps\,, \quad \forall \,\eps>0\,,
\end{flalign*}
which implies
\begin{equation}\label{c2.2}
\lim_{E\to\infty}\frac{1}{\log\sqrt{E}}\int_{1}^{\sqrt{E}\,\diam(D_1 \cap D_2)}\,\frac{1}{\psi}\,\area\(D_1\cap \(D_2^{+\psi/\sqrt{E}}\setminus D_2^{-\psi/\sqrt{E}}\)\)\,d\psi=0\,.
\end{equation}
Combining \eqref{c2.1} with \eqref{c2.2}, we can conclude that, as $E\sto\infty$,
$$
C_2 = o\(\frac{\log E}{E}\)\,.
$$
To deal with $C_3$, consider first the case in which $\diam(D_1\cap D_2)>0$. Exploiting again the asymptotic relations in \eqref{asymp_infinity}, we have
\begin{flalign*}
\abs{C_3}%&=\int_{\diam(D_1 \cap D_2)}^{\diam(D_1 \cup D_2)}\,d\phi \,\int_{D_1\cap D_2^{+\phi}}\,dx\int_{\partial B_{\phi}(x)\cap D_2}\,J_0\(2\pi\sqrt{E}\norm{x-y}\)^4\,dy\\
&=\abs{\int_{\diam(D_1 \cap D_2)}^{\diam(D_1 \cup D_2)}\,\phi\,d\phi \,\int_{D_1\cap D_2^{+\phi}}\,dx\int_{\partial B_{\phi}(x)\cap D_2}\,\prod_{i,j=0}^2\,\tilde{r}^E_{i,j}(\phi\cos\theta,\phi\sin\theta)^{q_{i,j}}\,d\theta}\\
&\leq\frac{2\pi}{E}\,\int_{\sqrt{E}\,\diam(D_1 \cap D_2)}^{\sqrt{E}\,\diam(D_1 \cup D_2)}\,\frac{1}{\pi^4\,\psi}\int_{D_1\cap D_2^{+\psi/\sqrt{E}}}\,dx\\
&\leq\frac{2}{\pi^3E}\,\area\(D_1\)\,\int_{\sqrt{E}\,\diam(D_1 \cap D_2)}^{\sqrt{E}\,\diam(D_1 \cup D_2)}\,\frac{1}{\psi}\,d\psi\\
&=\frac{2}{\pi^3\,E}\,\area\(D_1\)\,\(\log(\sqrt{E}\,\diam(D_1 \cup D_2))-\log(\sqrt{E}\,\diam(D_1 \cap D_2))\)\\
&=\log\[\frac{\diam(D_1 \cup D_2)}{\diam(D_1 \cap D_2)}\]\frac{2}{\pi^3\,E}\,\area\(D_1\)=O\(\frac{1}{E}\)\,.
\end{flalign*}
If $\diam(D_1\cap D_2) = 0$, then one proves exactly as above that 
\begin{flalign*}
& |C_3| = O\(\frac{1}{E}\)+\frac{c}{E}\int_{1}^{\sqrt E \diam(D_1 \cup D_2)}\,\frac 1\psi \,\area\(D_1\cap D_2^{+\psi/\sqrt E}\)\,d\psi\,.
\end{flalign*}
\noindent Now, since $\area\(D_1\cap D_2\)=0$ and $D_2$ has a $C^1$ boundary, as $\alpha\downarrow 0$,
$$
\area\(D_1\cap D_2^{+\alpha }\) \to 0\,;
$$
as before, this implies that $\forall\,\eps>0$,  there exists $\delta >0$ such that 
$$
\area\(D_1\cap D_2^{+\psi/\sqrt{E}}\) \leq \eps\,, \quad\text{whenever $\frac{\psi}{\sqrt{E}}<\delta$\,. }
$$
For any fixed $\eps>0$, pick such a $\delta$ to split the integral as follows
\begin{flalign*}
&\int_{1}^{\sqrt E \diam(D_1 \cup D_2)}\,\frac 1\psi \,\area\(D_1\cap D_2^{+\psi/\sqrt E}\)\,d\psi=\\
&=\underbracket[0.5pt][5pt]{\int_{1}^{\delta\sqrt{E}}\,\frac{1}{\psi}\,\area\(D_1\cap D_2^{+\psi/\sqrt E}\)\,d\psi}_{=:\, I_1}+\underbracket[0.5pt][5pt]{\int_{\delta\sqrt{E}}^{\sqrt{E}\,\diam(D_1 \cup D_2)}\,\frac{1}{\psi}\,\area\(D_1\cap D_2^{+\psi/\sqrt E}\)\,d\psi}_{=:\, I_2}\,;
\end{flalign*}
hence\footnote{Note that, if $\diam(D_1 \cup D_2)\le\delta$, there is no need of splitting the integral in the sum of $I_1$ and $I_2$, as $$\int_{1}^{\sqrt E \diam(D_1 \cup D_2)}\,\frac 1\psi\,d\psi \,\area\(D_1\cap D_2^{+\psi/\sqrt E}\)\le \int_{1}^{\delta \sqrt E}\,\frac 1\psi\,d\psi \,\area\(D_1\cap D_2^{+\psi/\sqrt E}\)=I_1\,.$$}
\begin{flalign*}
I_{2}&\leq \area\(D_1\)\, \log\(\frac{\diam(D_1 \cup D_2)}{\delta}\) \quad\text{while}\quad 
I_{1}\leq \eps\, \log\(\delta\sqrt{E}\)\,.
\end{flalign*}
As a consequence
\begin{flalign*}
&\limsup_{E\to\infty}\frac{1}{\log\sqrt{E}}\int_{1}^{\sqrt{E}\,\diam(D_1 \cup D_2)}\,\frac{1}{\psi}\,\area\(D_1\cap D_2^{+\psi/\sqrt{E}}\)\,d\psi\\
&\leq\limsup_{E\to\infty}\frac{1}{\log\sqrt{E}}\(\eps\, \log\(\delta\sqrt{E}\)+\area\(D_1\)\, \log\(\frac{\diam(D_1 \cup D_2)}{\delta}\)\)=\eps\,, \quad \forall \,\eps>0\,,
\end{flalign*}
which implies
\begin{equation}\label{i2}
\lim_{E\to\infty}\frac{1}{\log\sqrt{E}}\int_{1}^{\sqrt{E}\,\diam(D_1 \cup D_2)}\,\frac{1}{\psi}\,\area\(D_1\cap D_2^{+\psi/\sqrt{E}}\)\,d\psi=0\,.
\end{equation}
Therefore, as $E\sto\infty$, we have that
$$
\int_{D_1}\int_{D_2}\,\prod_{i,j=0}^2\,\widetilde{r}^E_{i,j}(x-y)^{q_{i,j}}\,dx\,dy=o\(\frac{\log E}{E}\)\,,
$$ 
and this concludes the proof.
\end{proof}

\medskip

\begin{prop}\label{prop2}
Let $q_{i,j}\geq 1$ for $i,j=0,1,2$ and $\sum_{i,j=0}^2\, q_{i,j}=4$. Then, for all $D_1,D_2 \in \A$ %are such that $area\(D_1 \cap D_2	\)>0$, 
we have, as $E\sto\infty$, 
\begin{flalign}
&\int_{D_1}\int_{D_2}\,\prod_{i,j=0}^2\,\widetilde{r}^E_{i,j}(x-y)^{q_{i,j}}\,dx\,dy=o\(\frac{\log E}{E}\)+\qquad\qquad\qquad\qquad\qquad\qquad\qquad\qquad\qquad\qquad\qquad\qquad\qquad\notag\\
&+\int_{0}^{2\pi}\,\prod_{i,j=0}^2\,h^1_{i,j}(\theta)^{q_{i,j}}\,d\theta\,\frac{1}{E}\,\int_{1}^{\sqrt{E}\diam(D_1 \cap D_2)}\area\(D_1\cap D_2^{-\psi/\sqrt{E}}\)g^1_{i,j}(\psi)^{q_{i,j}}\psi\,d\psi\,,
\end{flalign}
where the functions $g^E_{i,j}\,, \,h^E_{i,j}$ are defined in \eqref{asymp_infinity}.
\end{prop}
\proof
Performing a change of variable, we have that the first term on the right hand side of \eqref{eq_prop1} is equal to
\begin{flalign*}
&\frac{1}{E}\,\int_{0}^{\sqrt{E}\,\diam(D_1 \cap D_2)}\,d\psi\,\int_{D_1\cap D_2^{-\psi/\sqrt{E}}}\,dx \,\int_{0}^{2\pi}\,\psi\,d\theta\,\prod_{i,j=0}^2\,\tilde{r}^1_{i,j}(\psi\cos\theta,\psi\sin\theta)^{q_{i,j}}\\
&=\frac{1}{E}\,\int_{0}^{\sqrt{E}\,\diam(D_1 \cap D_2)}\,d\psi\,\area\(D_1\cap D_2^{-\frac{\psi}{\sqrt{E}}}\)\,\int_{0}^{2\pi}\,\psi\,d\theta\,\prod_{i,j=0}^2\,\tilde{r}^1_{i,j}(\psi\cos\theta,\psi\sin\theta)^{q_{i,j}}\\
&=\frac{1}{E}\,\int_{0}^{1}\,d\psi\,\area\(D_1\cap D_2^{-\psi/\sqrt{E}}\)\,\int_{0}^{2\pi}\,\psi\,d\theta\,\prod_{i,j=0}^2\,\tilde{r}^1_{i,j}(\psi\cos\theta,\psi\sin\theta)^{q_{i,j}}\\
&+\frac{1}{E}\,\int_{1}^{\sqrt{E}\,\diam(D_1 \cap D_2)}\,d\psi\,\area\(D_1\cap D_2^{-\psi/\sqrt{E}}\)\int_{0}^{2\pi}\,\psi\,d\theta\,\prod_{i,j=0}^2\,\tilde{r}^1_{i,j}(\psi\cos\theta,\psi\sin\theta)^{q_{i,j}} =:C_1\,.
\end{flalign*}
Using the asymptotic relations in \eqref{asymp_infinity} and \eqref{asymp_zero}, we have that
\begin{flalign*}
&C_1 = O\(\frac{1}{E}\)\\
& \quad\quad\quad +\int_{0}^{2\pi}\prod_{i,j=0}^2 h^1_{i,j}(\theta)^{q_{i,j}} d\theta\frac{1}{E}\int_{1}^{\sqrt{E}\diam(D_1 \cap D_2)}\area\(D_1\cap D_2^{-\psi/\sqrt{E}}\)g^1_{i,j}(\psi)^{q_{i,j}}\psi\,d\psi\,,
\end{flalign*}
which concludes the proof.
\endproof

\medskip

\begin{proof}[{\bf Proof of Theorem \ref{thm_L}, real case}]
From Theorem \ref{NPR-MR} suitably extended to the class $\mathscr{A}$ (see Remark \ref{r:papeete}), we already know that
$$
\widetilde{\L}_E(D)\stackrel{d}{\to}Z\sim\NN(0,1)\,, \text{ as } \  E\sto\infty\,,
$$
which is implied by the convergence of the fourth chaotic component $\L_{E}^{[4]}(D)$, that is the projection of $\L_E(D)$ onto the $4$-th Wiener chaos associated with $B_E$ (see Section A.2), i.e.
$$
\frac{\L_{E}^{[4]}(D)}{\sqrt{\Var\L_{E}^{[4]}(D)}}\stackrel{d}{\to}Z\sim\NN(0,1)\,,
\quad \text{  and the fact that } \quad
%\begin{equation}\label{fourth_dom}
\widetilde{\L}_E(D)=\frac{\L_{E}^{[4]}(D)}{\sqrt{\Var\L_{E}^{[4]}(D)}}+o(1)\,,
%\end{equation}
$$
where the notation $o(1)$ indicates a sequence converging to zero in $L^2(\P)$. Moreover, from Proposition \ref{propvar4L}, we infer that
\begin{flalign}\label{L_4_norm}
&\frac{\L_E^{[4]}(D)}{\sqrt{\Var \L_E^{[4]}(D)}}=\underbracket[0.5pt][5pt]{ \,\,\,\, \frac{\sqrt{\pi^3 \,E}}{\sqrt{16\,\area (D)\,\log E}} \,\,\,\, }_{=: \, K_E(D)}\,\times\notag\\
&\qquad\times\{8\,a_{1,E}(D)-a_{2,E}(D)-a_{3,E}(D)-2\,a_{4,E}(D)-8\,a_{5,E}(D)-8\,a_{6,E}(D)\}\,.
\end{flalign}
To prove the convergence of $\(\widetilde{\L}_E(D_1),\dots,\widetilde{\L}_E(D_m)\)$, $D_i \in \A$ for each $i=1,\dots,m$, one can now use \cite[Theorem 1]{PT:05}, which, since each variable $K_E(D)\,a_{i,E}(D)$, $i\in \{1,\dots,6\}$ is a member of the fourth Wiener chaos associated with $B_E$ and converges to a Gaussian random variable, requires us to show that each covariance $K_E(D_1)K_E(D_2) \Cov\(a_{i,E}(D_1), a_{j,E}(D_2)\)$, $1\leq i,j\leq 6$, converges to an appropriate constant, as $E\to\infty$. 

\medskip

If $\area\(D_1\cap D_2\)=0$, using Proposition \ref{prop1}, we have that 
$$
\Cov\(a_{1,E}(D_1), a_{1,E}(D_2)\)=24\,\int_{D_1}\int_{D_2}\,r^E(x-y)^4\,dx\,dy=o\(\frac{\log E} E\)\,;
$$
while if $\area\(D_1\cap D_2\)>0$, using Proposition \ref{prop2}, we have that
\begin{flalign*}
&\Cov\(a_{1,E}(D_1), a_{1,E}(D_2)\)=24\,\int_{D_1}\int_{D_2}\,r^E(x-y)^4\,dx\,dy\\
&=o\(\frac{\log E}{E}\)+\frac{2\pi\,24}{E}\,\int_{1}^{\sqrt{E}\,\diam(D_1 \cap D_2)}\,d\psi\,\area\(D_1\cap D_2^{-\psi/\sqrt{E}}\)\,\frac{1}{\pi^4\,\psi}\cos^4\(2\pi\,\psi-\frac{\pi}{4}\)\,.
\end{flalign*}
Recalling that
$$
\cos^4x=\frac{3}{8}+\frac{1}{8}\cos(4x)+\frac{1}{2}\cos(2x)\,,
$$
one has that
\begin{flalign}
&\Cov\(a_{1,E}(D_1), a_{1,E}(D_2)\)\notag\\
&= o\(\frac{\log E}{E}\)+\frac{2\pi\,24}{E}\,\int_{1}^{\sqrt{E}\,\diam(D_1 \cap D_2)}\,\area\(D_1\cap D_2^{-\psi/\sqrt{E}}\)\,\times\notag\\
&\qquad\times\,\frac{1}{\pi^4\,\psi}\[\frac{3}{8}+\frac{1}{8}\cos\(8\pi\,\psi-\pi\)+\frac{1}{2}\cos\(4\pi\,\psi-\frac{\pi}{2}\)\]\,d\psi\notag\\
&= o\(\frac{\log E}{E}\)+O\(\frac{1}{E}\)+\frac{18}{\pi^3\,E}\,\int_{1}^{\sqrt{E}\,\diam(D_1 \cap D_2)}\,\frac{1}{\psi}\,\area\(D_1\cap D_2^{-\psi/\sqrt{E}}\)\,d\psi\,, \label{c1.1}
\end{flalign} 
where the $O(E^{-1})$ term comes from integrating the cosines --- see Remark \ref{cosines} for more details. Moreover, as $\alpha\downarrow 0$ and since $D_2$ has a smooth boundary, 
$$
\area\(D_1\cap D_2^{-\alpha} \)\to \area\(D_1\cap D_2\),
$$
implying that that $\forall\,\eps>0$, there exists $\delta >0$ such that 
\begin{equation}\label{eps_delta}
\area\(D_1\cap D_2^{-\psi/\sqrt{E}}\)\geq \area\(D_1\cap D_2\) - \eps \,, \quad\text{whenever $\frac{\psi}{\sqrt{E}}<\delta$\,. }
\end{equation}
Now,
\begin{flalign}
&\int_{1}^{\sqrt{E}\,\diam(D_1 \cap D_2)}\,\frac{1}{\psi}\,\area\(D_1\cap D_2^{-\psi/\sqrt{E}}\)\,d\psi=\notag\\
&=\underbracket[0.5pt][5pt]{\int_{1}^{\delta\sqrt{E}}\,\frac{1}{\psi}\,\area\(D_1\cap D_2^{-\psi/\sqrt{E}}\)\,d\psi}_{C_{1,1}}+\underbracket[0.5pt][5pt]{\int_{\delta\sqrt{E}}^{\sqrt{E}\,\diam(D_1 \cap D_2)}\,\frac{1}{\psi}\,\area\(D_1\cap D_2^{-\psi/\sqrt{E}}\)\,d\psi}_{C_{1,2}} \label{smart_split}
\end{flalign}
and\footnote{Note that, if $\delta$ is such that $\diam(D_1 \cup D_2)\le\delta$, there is no need of splitting the integral in the sum of $C_{1,1}$ and $C_{2,2}$, as
$$\int_{1}^{\sqrt{E}\diam(D_1 \cap D_2)}\,\frac{1}{\psi}\,\area\(D_1\cap D_2^{-\psi/\sqrt{E}}\)\,d\psi \le \int_{1}^{\delta\sqrt{E}}\,\frac{1}{\psi}\,\area\(D_1\cap D_2^{-\psi/\sqrt{E}}\)\,d\psi= C_{1,1}$$}
\begin{flalign*}
C_{1,2}&\leq \area\(D_1\cap D_2\)\, \log\(\frac{\diam(D_1 \cap D_2)}{\delta}\)
\end{flalign*}
while
\begin{flalign*}
&\(\area\(D_1\cap D_2\) - \eps\) \,\log\(\delta\sqrt{E}\)\leq C_{1,1}\leq \area\(D_1\cap D_2\)\, \log\(\delta\sqrt{E}\)\,.
\end{flalign*}
Hence
\begin{flalign*}
&\limsup_{E\to\infty}\frac{1}{\log\sqrt{E}}\int_{1}^{\sqrt{E}\,\diam(D_1 \cap D_2)}\,\frac{1}{\psi}\,\area\(D_1\cap D_2^{-\psi/\sqrt{E}}\)\,d\psi\\
&\leq\limsup_{E\to\infty}\,\frac{1}{\log\sqrt{E}}\(\area\(D_1\cap D_2\)\, \log\(\delta\sqrt{E}\)+\area\(D_1\cap D_2\)\, \log\(\frac{\diam(D_1 \cap D_2)}{\delta}\)\)\\
&=\area\(D_1\cap D_2\)
\end{flalign*}
and
\begin{flalign*}
\liminf_{E\to\infty}\frac{1}{\log\sqrt{E}}\int_{1}^{\sqrt{E}\,\diam(D_1 \cap D_2)}\,\frac{1}{\psi}\,\area\(D_1\cap D_2^{-\psi/\sqrt{E}}\)\,d\psi&\geq\area\(D_1\cap D_2\)-\eps
\end{flalign*}
for each $\eps>0$, and consequently
\begin{equation}\label{c1.2}
\lim_{E\to\infty}\frac{1}{\log\sqrt{E}}\int_{1}^{\sqrt{E}\,\diam(D_1 \cap D_2)}\,\frac{1}{\psi}\,\area\(D_1\cap D_2^{-\psi/\sqrt{E}}\)\,d\psi=\area\(D_1\cap D_2\)\,.
\end{equation}
Therefore, combining \eqref{c1.1} with \eqref{c1.2}, we can conclude that, as $E \sto \infty$ and for $D_1,D_2$ such that $\area(D_1\cap D_2)>0$,
\begin{equation}\label{1strate}
\Cov\(a_{1,E}(D_1), a_{1,E}(D_2)\) \sim \area\(D_1\cap D_2\)\,\frac{9\,\log E}{\pi^3\,E}\,.
\end{equation}

\medskip

\begin{oss}\label{cosines}
Fix $0<\eps\lll1$, and let $\delta=\delta_\eps$ be as in \eqref{eps_delta}, then
\begin{flalign}
&\frac{6\pi}{E}\,\int_{1}^{\sqrt{E}\,\diam(D_1 \cap D_2)}\,\area\(D_1\cap D_2^{-\psi/\sqrt{E}}\)\,\frac{\cos\(8\pi\,\psi-\pi\)}{\psi}\,d\psi\notag\\
&=\frac{6\pi}{E}\,\int_{1}^{\delta\sqrt{E}}\,\area\(D_1\cap D_2^{-\psi/\sqrt{E}}\)\,\frac{\cos\(8\pi\,\psi-\pi\)}{\psi}\,d\psi\notag\\
&\quad+\frac{6\pi}{E}\,\int_{\delta\sqrt{E}}^{\sqrt{E}\,\diam(D_1 \cap D_2)}\,\area\(D_1\cap D_2^{-\psi/\sqrt{E}}\)\,\frac{\cos\(8\pi\,\psi-\pi\)}{\psi}\,d\psi\,.\notag
\end{flalign}
Now,
\begin{flalign}
&\frac{6\pi}{E}\,\int_{1}^{\delta\sqrt{E}}\,\area\(D_1\cap D_2^{-\psi/\sqrt{E}}\)\,\frac{\cos\(8\pi\,\psi-\pi\)}{\psi}\,d\psi\\
&\sim \frac{6\pi}{E}\,\area\(D_1\cap D_2\)\,\int_{1}^{\delta\sqrt{E}}\,\,\frac{\cos\(8\pi\,\psi-\pi\)}{\psi}\,d\psi\notag\\
&= \frac{3}{E}\,\area\(D_1\cap D_2\)\,\{\[\frac{\sin\(8\pi\,\psi-\pi\)}{\psi}\]_{1}^{\delta\sqrt{E}}+\int_{1}^{\delta\sqrt{E}}\,\frac{\sin\(8\pi\,\psi-\pi\)}{\psi^2}\,d\psi\}\notag\\
&=O\(\frac{1}{E^{3/2}}\)\,,\label{okkk}
\end{flalign}
while
\begin{flalign}
&\frac{6\pi}{E}\,\int_{\delta\sqrt{E}}^{\sqrt{E}\,\diam(D_1 \cap D_2)}\,\area\(D_1\cap D_2^{-\psi/\sqrt{E}}\)\,\frac{\abs{\cos\(8\pi\,\psi-\pi\)}}{\psi}\,d\psi\notag\\
&\leq\frac{6\pi}{E}\,\area\(D_1\cap D_2\)\,\int_{\delta\sqrt{E}}^{\sqrt{E}\,\diam(D_1 \cap D_2)}\,\frac{1}{\psi}\,d\psi\notag\\
&=\frac{6\pi}{E}\,\area\(D_1\cap D_2\)\[\log\(\sqrt{E}\,\diam(D_1 \cap D_2)\)-\log\(\delta\sqrt{E}\)\]\notag\\
&=\frac{6\pi}{E}\,\area\(D_1\cap D_2\)\,\log\(\frac{\diam(D_1 \cap D_2)}{\delta}\)=O\(\frac{1}{E}\)\,,\label{okkk2}
\end{flalign}
which explains why the two oscillating terms in \eqref{c1.1} are negligible (the second oscillating term is treated exactly in the same manner).

\end{oss}

\medskip

\noindent Whenever $\area\(D_1\cap D_2\)>0$, we can proceed in a completely analogous way to obtain the following rates (i.e. applying Proposition \ref{prop2} and splitting the integral as in \eqref{smart_split}), as $E\sto\infty$:
\begin{flalign*}
\Cov\(a_{1,E}(D_1), a_{2,E}(D_2)\)&=24 \,\int_{D_1}\int_{D_2}\,\widetilde{r}_{0,1}^E(x-y)^4\,dx\,dy \sim \area\(D_1\cap D_2\)\,\frac{27\,\log E}{2\,\pi^3\,E}\\
\Cov\(a_{1,E}(D_1), a_{3,E}(D_2)\)&=24 \,\int_{D_1}\int_{D_2}\,\widetilde{r}_{0,2}^E(x-y)^4\,dx\,dy \sim \area\(D_1\cap D_2\)\,\frac{27\,\log E}{2\,\pi^3\,E}\\
\Cov\(a_{1,E}(D_1), a_{4,E}(D_2)\)&=24 \,\int_{D_1}\int_{D_2}\,\widetilde{r}_{0,1}^E(x-y)^2\,\widetilde{r}_{0,2}^E(x-y)^2\,dx\,dy\\
&\sim \area\(D_1\cap D_2\)\frac{9\,\log E}{2\,\pi^3\,E}\\
\Cov\(a_{1,E}(D_1),a_{5,E}(D_2)\)&=24\int_{D_1}\int_{D_2}r^E(x-y)^2\widetilde r_{0,1}^E(x-y)^2\,dxdy\\
&\sim \area(D_1\cap D_2)\frac{3\log E}{\pi^3E},\\
\Cov\(a_{1,E}(D_1),a_{6,E}(D_2)\)&=24\int_{D_1}\int_{D_2}r^E(x-y)^2\widetilde r_{0,2}^E(x-y)^2\,dxdy\\
&\sim \area(D_1\cap D_2)\frac{3\log E}{\pi^3E},\\
\Cov\(a_{2,E}(D_1), a_{2,E}(D_2)\)&=24\int_{D_1}\int_{D_2}\widetilde r_{1,1}^E(x-y)^4 \,dxdy\sim\area(D_1\cap D_2) \frac{315\log E}{8\pi^3E},\\
\Cov\(a_{2,E}(D_1), a_{3,E}(D_2)\)&=24\int_{D_1}\int_{D_2}\widetilde r_{1,2}^E(x-y)^4\,dxdy\sim\area(D_1\cap D_2) \frac{27\log E}{8\pi^3E},\\
\Cov\(a_{2,E}(D_1), a_{4,E}(D_2)\)&=24\int_{D_1}\int_{D_2}\widetilde r_{1,1}^E(x-y)^2\widetilde r_{1,2}^E(x-y)^2\,dxdy\\
&\sim\area(D_1\cap D_2)\frac{45\log E}{8\pi^3E},\\
\Cov\(a_{2,E}(D_1), a_{5,E}(D_2)\)&=24\int_{D_1}\int_{D_2}\widetilde r_{0,1}^E(x-y)^2\widetilde r_{1,1}^E(x-y)^2\,dxdy\\
&\sim\area(D_1\cap D_2)\frac{15\log E}{2\pi^3E},\\
\Cov\(a_{2,E}(D_1), a_{6,E}(D_2)\)&=24\int_{D_1}\int_{D_2}\widetilde r_{0,1}^E(x-y)^2\widetilde r_{1,2}^E(x-y)^2\,dxdy\\
&\sim\area(D_1\cap D_2) \frac{3\log E}{2\pi^3E},\\
\Cov\(a_{3,E}(D_1), a_{3,E}(D_2)\)&=24\int_{D_1}\int_{D_2}\widetilde r_{2,2}^E(x-y)^4\,dxdy\sim\area(D_1\cap D_2)\frac{315\log E}{8\pi^3E},\\
\Cov\(a_{3,E}(D_1), a_{4,E}(D_2)\)&=24 \int_{D_1}\int_{D_2}\widetilde r_{2,2}^E(x-y)^2\widetilde r_{1,2}^E(x-y)^2 \,dxdy\\
&\sim\area(D_1\cap D_2) \frac{45\log E}{8\pi^3E},\\
\Cov\(a_{3,E}(D_1), a_{5,E}(D_2)\)&=24 \int_{D_1}\int_{D_2}\widetilde r_{0,2}^E(x-y)^2\widetilde r_{1,2}^E(x-y)^2\,dxdy\\
&\sim\area(D_1\cap D_2) \frac{3\log E}{2\pi^3E},\\
\Cov\(a_{3,E}(D_1), a_{6,E}(D_2)\)&=24\int_{D_1}\int_{D_2}\widetilde r_{0,2}^E(x-y)^2\widetilde r_{2,2}^E(x-y)^2\,dxdy\\
&\sim\area(D_1\cap D_2) \frac{15\log E}{2\pi^3E},\\
\Cov\(a_{4,E}(D_1), a_{4,E}(D_2)\)&=4\int_{D_1}\int_{D_2}(\widetilde r_{1,1}^E(x-y)^2\widetilde r_{2,2}^E(x-y)^2+\widetilde r_{1,2}^E(x-y)^4\\
&\qquad+4\widetilde r_{1,1}^E(x-y) \widetilde r_{2,2}^E(x-y)\widetilde r_{1,2}^E(x-y)^2)\,dxdy\\
&\sim\area(D_1\cap D_2) \frac{27\log E}{8\pi^3E},\\
\Cov\(a_{4,E}(D_1), a_{5,E}(D_2)\)&=4\int_{D_1}\int_{D_2}(\widetilde r_{0,1}^E(x-y)^2\widetilde r_{1,2}^E(x-y)^2+\widetilde r_{0,2}^E(x-y)^2\widetilde r_{1,1}^E(x-y)^2\\
&\qquad+4\widetilde r_{0,1}^E(x-y)\widetilde r_{0,2}^E(x-y)\widetilde r_{1,1}^E(x-y)\widetilde r_{1,2}^E(x-y))\,dxdy\\
&\sim\area(D_1\cap D_2)\frac{3\log E}{2\pi^3E},\\
\Cov\(a_{4,E}(D_1), a_{6,E}(D_2)\)&=4\int_{D_1}\int_{D_2}(\widetilde r_{0,1}^E(x-y)^2\widetilde r_{2,2}^E(x-y)^2+\widetilde r_{0,2}^E(x-y)^2\widetilde r_{1,2}^E(x-y)^2\\
&\qquad+ 4\widetilde r_{0,1}^E(x-y)\widetilde r_{0,2}^E(x-y)\widetilde r_{2,2}^E(x-y)\widetilde r_{1,2}^E(x-y))\,dxdy\\
&\sim\area(D_1\cap D_2) \frac{3\log E}{2\pi^3E},\\
\Cov\(a_{5,E}(D_1), a_{5,E}(D_2)\)&=4\int_{\mathcal{D}}\int_{\mathcal D}\big(r^E(x-y)^2\widetilde r_{1,1}^E(x-y)^2+\widetilde r^E_{0,1}(x-y)^4\\
&\qquad-4r^E(x-y)\widetilde r^E_{1,1}(x-y)\widetilde r_{0,1}^E(x-y)^2\big)\,dxdy\\
&\sim\area(D_1\cap D_2) \frac{3\log E}{2\pi^3E},\\
\Cov\(a_{5,E}(D_1), a_{6,E}(D_2)\)&= 4\int_{\mathcal{D}}\int_{\mathcal D}(r^E(x-y)^2\widetilde r_{2,2}^E(x-y)^2+\widetilde r^E_{0,2}(x-y)^4\\
&\qquad-4r^E(x-y)\widetilde r_{0,2}^E(x-y)^2 \widetilde r^E_{2,2}(x-y))\,dxdy\\
&\sim\area(D_1\cap D_2) \frac{\log E}{2\pi^3E},\\
\Cov\(a_{6,E}(D_1), a_{6,E}(D_2)\)&=4\int_{\mathcal{D}}\int_{\mathcal D}\big(r^E(x-y)^2\widetilde r_{2,2}^E(x-y)^2+\widetilde r^E_{0,2}(x-y)^4\\
&\qquad- 4r^E(x-y) \widetilde r^E_{2,2}(x-y)\widetilde r_{0,2}^E(x-y)^2\big)\,dxdy\\
&\sim\area(D_1\cap D_2) \frac{3\log E}{2\pi^3E}.
\end{flalign*}
On the other hand, when $\area\(D_1\cap D_2\)=0$, applying Proposition \ref{prop1}, one has that, for all $1\leq i,j\leq 6$ 
$$\Cov\(a_{i,E}(D_1), a_{j,E}(D_2)\)=o\(\frac{\log E}E\)\,.$$
Thus, we just obtained that each term
$$
K_E(D_1)K_E(D_2) \,\Cov\(a_{i,E}(D_1), a_{j,E}(D_2)\)\,, \quad 1\leq i,j\leq 6 \,,
$$ 
converges to a constant, as $E \sto \infty$ (where $K_E(D)$ is defined in \eqref{L_4_norm}). Since each variable $K_E(D_a)a_{i,E}(D_a)$, $a=1,2$, is a member of the fourth Wiener chaos associated with $B_E$ and, as $E\sto \infty$, each of them converges in distribution to a Gaussian random variable, \cite[Theorem 1]{PT:05} implies that the vector $(K_E(D_l)a_{i,E}(D_l) : i=1,...,6,\, l=1,2)$ converges in distribution to a centered Gaussian vector. Moreover, this implies that for any $m\ge1$, also $(\widetilde{\L}_E(D_1),\dots,\widetilde{\L}_E(D_m))$ converges to a Gaussian vector and the covariance structure of our limit object is obtained by a direct computation:
\begin{flalign*}
&\Cov\(\widetilde{\L}_E(D_1),\widetilde{\L}_E(D_2) \)\sim \Cov\(\frac{\L_{E}^{[4]}(D_1)}{\sqrt{\Var\L_{E}^{[4]}(D_1)}},\frac{\L_{E}^{[4]}(D_2)}{\sqrt{\Var\L_{E}^{[4]}(D_2)}}\)\\
&=\frac{\pi^3\,E}{16\,\sqrt{\area(D_1)\area(D_2)}\,\log E}\,\times\\
&\quad \times\,\Cov\Big(8\,a_{1,E}(D_1)-a_{2,E}(D_1)-a_{3,E}(D_1)-2\,a_{4,E}(D_1)-8\,a_{5,E}(D_1)-8\,a_{5,E}(D_1),\\
& \qquad\qquad\quad 8\,a_{1,E}(D_2)-a_{2,E}(D_2)-a_{3,E}(D_2)-2\,a_{4,E}(D_2)-8\,a_{5,E}(D_2)-8\,a_{5,E}(D_2)\Big)\\
&\longrightarrow 
%\begin{cases}\displaystyle{\frac{\area (D_1\cap D_2)}{\sqrt{\area(D_1)\area(D_2)}}} &\text{ if } \area\(D_1\cap D_2\)>0 \\
%0 & \displaystyle{\text{ if } \area\(D_1\cap D_2\)=0}
%\end{cases}=
\frac{\area (D_1\cap D_2)}{\sqrt{\area(D_1)\area(D_2)}}\	,
\end{flalign*}
as $E \to \infty$.
\end{proof}

\medskip

\begin{proof}[{\bf Proof of Theorem \ref{thm_L}, complex case}]
Also in this case we know from Theorem \ref{NPR-MR} (suitably extended to the class $\mathscr{A}$ -- see Remark \ref{r:papeete}) that, as $E\sto\infty$,
$$
\widetilde{\NNN}_E(D)\stackrel{d}{\to}Z\sim\NN(0,1)\,,
$$
which is implied by the convergence of the fourth chaotic component, that is the projection of $\NNN_E(D)$ onto the $4$-th Wiener chaos associated with $B_E^\C$ (see Section A.2), i.e.
$$
\frac{\NNN_{E}^{[4]}(D)}{\sqrt{\Var\NNN_{E}^{[4]}(D)}}\stackrel{d}{\to}Z\sim\NN(0,1)\,,
\quad \text{  and the fact that } \quad
\widetilde{\NNN}_E(D)=\frac{\NNN_{E}^{[4]}(D)}{\sqrt{\Var\NNN_{E}^{[4]}(D)}}+o(1)\,,
$$
where once again $o(1)$ indicates a sequence converging to zero in $L^2(\P)$. Moreover, from Proposition \ref{propvar4N}, we have that
\begin{flalign}\label{N_4_norm}
&\frac{\NNN_E^{[4]}(D)}{\sqrt{\Var \NNN_E^{[4]}(D)}}=\underbracket[0.5pt][5pt]{\frac{4\pi^2\,E}{\sqrt{11\,\area (D)\,E\,\log E}}}_{C_E(D)}\,\{a_{E}(D)+\widehat{a}_{E}(D)+b_{E}(D)\}\,.
\end{flalign}
where $a_{E}=a_{E}(D),\,\widehat{a}_{E}=\widehat{a}_{E}(D),\,b_{E}=b_{E}(D)$ are uncorrelated and\footnote{Recall that $\widehat{a}_{E}(D)$ is defined in the same way as $a_{E}(D)$, except for the fact that one uses $\widehat{B}_E$ instead of $B_E$.}
$$
a_E=\frac{1}{8}\{8\,a_{1,E}-a_{2,E}-2a_{3,E}-8\,a_{4,E}\}\,
$$ 
$$
b_E = \{ 2b_{1,E} - b_{2,E} - b_{3,E} -b_{4,E} - b_{5,E} - \frac14 b_{6,E} -\frac14 b_{7,E} + \frac54 b_{8,E} + \frac54 b_{9,E}- 3b_{10,E} \}\,.
$$
In order to prove the convergence of the vector $\big(\widetilde{\NNN}_E(D_1),\dots,\widetilde{\NNN}_E(D_m)\big)$, $D_1, \dots, D_m \in \A$, we want to use once again \cite[Theorem 1]{PT:05}; namely, since we know that also each $C_E(D)b_{i,E}(D)$, $i=1,\dots,10$, is a member of the fourth Wiener chaos associated with $B_E$  and converges to a Gaussian random variable as $E\sto\infty$ (see \cite[Proposition 8.2]{NPR:19}) and since we already showed that $C_E(D_1)C_E(D_2) \,\Cov\(a_{i,E}(D_1), a_{j,E}(D_2)\)$, $1\leq i,j\leq 6$ converge to constants as $E\sto \infty$ (as $C_E(D)= 16\sqrt{\pi}\,K_E(D)/\sqrt{11}$), we just have to prove that also the covariances $C_E(D_1)C_E(D_2) \Cov\(b_{i,E}(D_1), b_{j,E}(D_2)\)$, $1\leq i, j\leq 10$, converge to some constants, as $E\sto \infty$. 

\medskip

Now, it is tedious but easy to show (one has to do analogous computations as for achieving \eqref{1strate}), that, for each $1\leq i \leq j \leq 10$, whenever $\area(D_1\cap D_2)>0$,
$$
\Cov\(b_{i,E}(D_1),b_{j,E}(D_2)\)\sim \frac{n_{i,j}}{64}\,\area (D_1\cap D_2)\frac{\log E}{\pi^3E}\,,\\
$$
where $n_{i,j}=n_{j,i}$ and 
\begin{equation}
n_{i,j}=\begin{cases} 
\,4 & \text{ if } \quad (i,j) \in \{\begin{split} & (2,7), (2,8), (2,9), (2,10), (3,6), (3,8), (3,9), (3,10), \\
&(4,7), (4,8), (4,9), (4,10), (5,6), (5,8), (5,9), (5,10) \end{split}\}\\
\,8 & \text{ if } \quad (i,j) \in \{(1,2), (1,3), (1,4), (1,5)\}\\
\,9 &\text{ if } \quad (i,j) \in \{(6,7), (8,8), (8,9), (8,10), (9,9), (9,10), (10,10)\}\\
\,12 &\text{ if } \quad (i,j) \in \{(1,8), (1,9), (1,10), (2,3), (2,5), (3,4), (4,5)\}\\
\,15 & \text{ if } \quad (i,j) \in \{(6,8),(6,9),(6,10), (7,8), (7,9), (7,10)\}\\
\,20 & \text{ if } \quad (i,j) \in \{(2,6), (3,7), (4,6), (5,7)\}\\
\,24 & \text{ if } \quad (i,j) =(1,1)\\
\,36 & \text{ if } \quad (i,j) \in \{(1,6), (1,7), (2,2), (2,4), (3,3), (3,5),(4,4), (5,5)\}\\
\,105 & \text{ if } \quad (i,j) \in \{(6,6), (7,7)\}\\
\end{cases}\,;
\end{equation}
on the other hand, whenever $\area(D_1\cap D_2)=0$, 
$$
\Cov\(b_{i,E}(D_1),b_{j,E}(D_2)\)=o\(\frac{\log E}{E}\)\,.
$$

Thus, we just obtained that each term
\begin{equation*}
\begin{split}
&C_E(D_1)C_E(D_2) \,\Cov\(a_{i,E}(D_1), a_{j,E}(D_2)\)\,, \quad 1\leq i,j\leq 6 \,, \\
& C_E(D_1)C_E(D_2) \,\Cov\(b_{i,E}(D_1), b_{j,E}(D_2)\)\,, \quad 1\leq i,j\leq 10 \,,
\end{split}
\end{equation*}
converges to a constant, as $E \sto \infty$. 
In conclusion, since each variable $C_E(D_l)a_{i,E}(D_l)$ and $C_E(D_l)b_{i,E}(D_l)$, $l=1,2$, is a member of the fourth Wiener chaos associated with $B_E$ and, as $E\sto \infty$, each of them converges in distribution to a Gaussian random variable, \cite[Theorem 1]{PT:05} implies that the vector $(C_E(D_l)a_{i,E}(D_l), C_E(D_h)b_{j,E}(D_h)  : i=1,\dots,6,\, j=1,\dots,10,\,  l,h=1,2)$ converges in distribution to a centered Gaussian vector. Moreover, this implies that, for any $m\ge1$, also $(\widetilde{\NNN}_E(D_1),\dots, \widetilde{\NNN}_E(D_m))$ converges to a Gaussian vector and the covariance structure of our limit object is obtained by a direct computation:
\begin{flalign*}
&\Cov\(\widetilde{\NNN}_E(D_1),\widetilde{\NNN}_E(D_2)\)\sim\Cov\(\frac{\NNN_{E}^{[4]}(D_1)}{\sqrt{\Var\NNN_{E}^{[4]}(D_1)}},\frac{\NNN_{E}^{[4]}(D_2)}{\sqrt{\Var\NNN_{E}^{[4]}(D_2)}}\)\\
&=\frac{16 \pi^4\,E}{11\,\sqrt{\area(D_1)\area(D_2)}\,\log E}\,\times \\
&\qquad\times\,\Cov\Big(a_{E}(D_1)+\widehat a_{E}(D_1)+b_{E}(D_1),\, a_{E}(D_2)+\widehat a_{E}(D_2)+b_{E}(D_2)\Big)\\
&\longrightarrow 
%\begin{cases}\displaystyle{\frac{\area (D_1\cap D_2)}{\sqrt{\area(D_1)\area(D_2)}}} &\text{ if } \area\(D_1\cap D_2\)>0 \\
%0 & \displaystyle{\text{ if } \area\(D_1\cap D_2\)=0}
%\end{cases}=
\frac{\area (D_1\cap D_2)}{\sqrt{\area(D_1)\area(D_2)}}\ ,
\end{flalign*}
as $E \to \infty$.
\end{proof}

%%%%%%%%%%%%%%%%%%%%%%%%------------------------------------------------------------------------------------------------------------------------------------------------------------------------------------------------------------------------------------------------------

\subsection{Proof of Theorem \ref{tightness}}\label{ss:prooftightness}

Recall the definitions of $X_E(t_1,t_2)$ and $Y_E(t_1,t_2)$ from Section 3.2; Theorem \ref{thm_L} straightforwardly implies that $X_E(t_1,t_2)$ and $Y_E(t_1,t_2)$ converge, as $E\sto\infty$ and in the sense of finite-dimensional distributions, to a 2-dimensional Wiener sheet, namely a centered Gaussian process 
$$\mathbf{W}=\{W(t_1,t_2): (t_1,t_2)\in[0,1]^2\}$$ 
with covariance function $E\[W(t_1,t_2)W(s_1,s_2)\]=\(t_1 \wedge s_1\)\,\(t_2 \wedge s_2\)$. Hence, in order to obtain a weak convergence result for $(X_E^{[4]}( \cdot ))_E$ (respectively $(Y_E^{[4]}( \cdot ))_E$), it is enough to prove that the sequence $(X_E^{[4]}( \cdot ))_E$ (respectively $(Y_E^{[4]}( \cdot ))_E$) is tight. We will do it by showing that $X_E^{[4]}(t_1,t_2)$ (respectively $Y_E^{[4]}(t_1,t_2)$) satisfies a Kolmogorov continuity criterion, i.e. that the following holds
\begin{equation}\label{Kol_criterion}
\begin{split}&\E[ (X_E^{[4]}(t_1,t_2) - X_E^{[4]}(s_1,s_2))^a]\leq K \|(t_1,t_2)-(s_1,s_2)\|^{2+b}\\&\E[ (Y_E^{[4]}(t_1,t_2) - Y_E^{[4]}(s_1,s_2))^a]\leq K \|(t_1,t_2)-(s_1,s_2)\|^{2+b}
\end{split}\,, \quad \text{ for some } \quad a,b >0
\end{equation}
and with $K$ an absolute finite constant\footnote{see also \cite[Theorem 2.1]{RY:99}.} ($\norm{\cdot}$ denotes the Euclidean norm on $\R^2$).

Let us start with $X_E^{[4]}(t_1,t_2)$. Without loss of generality (see Remark \ref{WLOG}), 
 assume that $s_1\leq t_1$ and $s_2\leq t_2$, then
\begin{flalign*}
&\E\[\abs{X_E^{[4]}(t_1,t_2)-X_E^{[4]}(s_1,s_2)}^2\]=\E\[\abs{\frac{\L_E^{[4]}([0,t_1]\times[0,t_2])-\L_E^{[4]}([0,s_1]\times[0,s_2])}{\sqrt{\frac{\log E}{512\pi}}}}^2\]	\\
&=\E\[\abs{\frac{\L_E^{[4]}(D_{t,s})}{\sqrt{\frac{\log E}{512\pi}}}}^2\]\leq \frac{\pi^2\,E}{4^2\,\log E}\E\[8^2\abs{a_{1,E}(D_{t,s})}^2+\abs{a_{2,E}(D_{t,s})}^2+\abs{a_{3,E}(D_{t,s})}^2+\right.\\
&\qquad\qquad\qquad\qquad\qquad\qquad\left.+2^2\abs{a_{4,E}(D_{t,s})}^2+8^2\abs{a_{5,E}(D_{t,s})}^2+8^2\abs{a_{5,E}(D_{t,s})}^2\]\,,
\end{flalign*}
where $D_{t,s}:=[0,t_1]\times[0,t_2]\setminus[0,s_1]\times[0,s_2]$.

Set $t:=(t_1,t_2)$ and $s:=(s_1,s_2)$. In the sequel, the letter $c$ will denote any positive constant that depends neither on $t,s$ nor on $E$. 
Thanks to the {\it diagram formula} (see \cite[Proposition 4.15]{MP:11}) and adapting Propositions \ref{prop1} and \ref{prop2} for $D_1=D_2=D_{t,s}$ (see Remark \ref{oss-tight}), we have
\begin{flalign*} 
\E\[\abs{a_{1,E}(D_{t,s})}^2\] &=\E\[\abs{\int_{D_{t,s}}\,H_4\(B_E(x)\)dx}^2\]=\int_{D_{t,s}^2}\,\E\[H_4\(B_E(x)\)H_4\(B_E(y)\)\]\,dx\, dy\\
&= 6 \,\int_{D_{t,s}^2}\, \E\[B_E(x)B_E(y)\]^4 \,dx\, dy= 6 \,\int_{D_{t,s}^2}\, r^E(x-y)^4\,dx\, dy\\
&\le c\, \area\(D_{t,s}\) \frac{\log E}{E} \leq c\,\norm{t-s} \frac{\log E}{E}\,.
\end{flalign*}
\begin{oss}\label{oss-tight}
Recall the proof of Proposition \ref{prop1}; using the coarea formula we have, for any $t,s \in [0,1]^2$
\begin{flalign*} 
&\int_{D_{t,s}^2}\, r^E(x-y)^4\,dx\, dy=\int_0^{\diam(D_{t,s})}\phi \, d\phi\int_{D_{t,s}^{-\phi}}dx\int_0^{2\pi}r^E(\phi\cos\theta,\phi\sin\theta)^4d\theta\\
&\quad+\int_0^{\diam(D_{t,s})}\phi \, d\phi\int_{D_{t,s}\setminus D_{t,s}^{-\phi}}dx\int_{\partial B_\phi(x)\cap D_{t,s}}r^E(\phi\cos\theta,\phi\sin\theta)^4d\theta\\
&\le\frac{1}{E}\int_0^{\sqrt E\diam(D_{t,s})}\psi \, d\psi\int_{D_{t,s}^{-\psi/\sqrt E}}dx\int_0^{2\pi}r^1(\psi\cos\theta,\psi\sin\theta)^4d\theta\\
&\quad+\frac{1}{E}\int_0^{\sqrt E\diam(D_{t,s})}\psi \, d\psi\int_{D_{t,s}\setminus D_{t,s}^{-\psi/\sqrt E}}dx\int_0^{2\pi}r^1(\psi\cos\theta,\psi\sin\theta)^4d\theta\\
&\le\frac{\area(D_{t,s})}{E}\int_0^{\sqrt E\diam(D_{t,s})}\psi \, d\psi \int_0^{2\pi}r^1(\psi\cos\theta,\psi\sin\theta)^4d\theta\\
&\le \frac{c \,\area(D_{t,s})}{E}\(\int_0^1\psi \,d\psi+\int_1^{\sqrt E\diam(D_{t,s})}\frac 1 \psi \, d\psi \) \le c \,\area(D_{t,s}) \, \frac{\log E}{E} \,,
\end{flalign*}
where we used once again the asymptotic relations for Bessel functions \eqref{asymp_zero} and \eqref{asymp_infinity}.
\end{oss}

Consequently, using the hypercontractivity property of functionals living in a fixed Wiener-chaos (see \cite[Theorem 2.7.2]{NP:12}), we have that
$$
\frac{64\,\pi^6\,E^3}{(\log E)^3}\E\[\abs{a_{1,E}(D_{t,s})}^6\]\leq c\, \norm{t-s}^3
\,.
$$
Moreover, one can prove in an analogous way that
$$
\frac{64\,\pi^6\,E^3}{(\log E)^3}\E\[\abs{a_{i,E}(D_{t,s})}^6\]\leq c\, \norm{t-s}^3
\,, 
$$
for each $i=2,3,4,5,6$. Therefore, we obtain that
\begin{flalign}\label{KOL_4th}
&\E\[\abs{X_E^{[4]}(t)-X_E^{[4]}(s)}^6\] \leq c \,\norm{t-s}^3
\end{flalign}
and hence that $X_E^{[4]}(t)$ satisfies  \eqref{Kol_criterion}, with $a=6$ and $b=1$. 
Thanks to the Kolmogorov continuity criterion for tightness, we just showed that $X_E^{[4]}(t)$ is tight. Showing the tightness of $Y_E^{[4]}(t)$ is completely analogous and it is left to the interested reader.
 
\endmproof

\begin{oss}\label{WLOG}
The reason why, taking $s_1\leq t_1$ and $s_2\leq t_2$, we do not loose any generality relies simply on the fact that the fourth chaotic component $X_E^{[4]}(t_1,t_2)$ is an integral over a domain $D_t:=[0,t_1]\times [0,t_2]$ and hence one can use additivity. More specifically, assume instead that $s_1\leq t_1$ but $s_2\geq t_2$, then 
$$
D_t=[0,s_1]\times [0,t_2]+[s_1,t_1]\times [0,t_2]
$$
and
$$
D_s=[0,s_1]\times [0,t_2]+[0,s_1]\times [t_2,s_2]\,.
$$
Consequently, doing analogous computations as the ones we used to reach equation \eqref{KOL_4th}, we have that 
\begin{flalign*}
&\E\[\abs{X_E^{[4]}(t_1,t_2)-X_E^{[4]}(s_1,s_2)}^6\]=\E\[\abs{\frac{\L_E^{[4]}\([s_1,t_1]\times [0,t_2]\)-\L_E^{[4]}\([0,s_1]\times [t_2,s_2]\)}{\sqrt{\frac{\log E}{512\pi}}}}^6\]\\
&\leq 2^6 \(\frac{512\pi}{\log E}\)^3 E\[\abs{\L_E^{[4]}\([s_1,t_1]\times [0,t_2]\)}^6+\abs{\L_E^{[4]}\([0,s_1]\times [t_2,s_2]\)}^6\]\\
&\leq c \{\area\([s_1,t_1]\times [0,t_2]\)^3+\area\([0,s_1]\times [t_2,s_2]\)^3\}\\
&= c \, \{\[(t_1-s_1)t_2\]^3+\[(s_2-t_2)s_1\]^3\}\leq c \, \{\abs{t_1-s_1}^3+\abs{s_2-t_2}^3\} \le c \norm{t-s}^3\,.
\end{flalign*}
\end{oss}

%%%%%%%%%%%%%%%%%%%%%%%%%%%%%%%%%%%%%%%%%%%%%%%%%%%%%%%%%%%%%%%%%%%%%%%%%%%%%%%%%%%%%%%%%%%%%%%%%%%%%%%%%%%%%%%%%%%%

\appendix

\section{Ancillary results from \cite{NPR:19} and more}\label{appA}

\subsection{Covariances}

In \cite[Lemma 3.1]{NPR:19}, the authors computed the distribution of the Gaussian vector $(B_E(x), B_E(y), \nabla B_E(x), \nabla B_E(y))\in \R^6$ for $x, y\in \R^2$, where $\nabla B_E$ is the gradient field $\nabla := (\partial_1, \partial_2), \partial_i := \partial_{x_i} = \partial/\partial {x_i}$ for $i=1,2$. For $i,j\in \lbrace 0,1,2 \rbrace$ define
\begin{equation}\label{notationDerivative}
r^E_{i,j}(x-y) := \partial_{x_i} \partial_{y_j} r^E(x-y),
\end{equation}
with $\partial_{x_0}$ and $\partial_{y_0}$ equal to the identity by definition.

\medskip 

\begin{lem}[{\cite[Lemma 3.1]{NPR:19}}]\label{lemsmooth}
The centered Gaussian vector
$$
(B_E(x), B_E(y), \nabla B_E(x), \nabla B_E(y))\in \R^6\,, \quad x\neq y\in \R^2\,,
$$  
has the following covariance matrix:
\begin{equation}\label{sigma}
\Sigma^E (x-y)= \begin{pmatrix}
\Sigma_{1}^E(x-y) &\Sigma_{2}^E(x-y)\\
\Sigma_{2}^E(x-y)^t &\Sigma_{3}^E(x-y)
\end{pmatrix},
\end{equation}
where 
$$
\Sigma_1^E(x-y) = \begin{pmatrix}
1 &r^E(x-y)\\
r^E(x-y) &1
\end{pmatrix},
$$
$r^E$ being defined in \eqref{covE}, 
\begin{equation}\label{matrixB}
\Sigma_2^E(x-y) = \begin{pmatrix}
0 &0 &r_{0,1}^E(x-y) &r_{0,2}^E(x-y)\\
-r_{0,1}^E(x-y) &-r_{0,2}^E(x-y) &0 &0
\end{pmatrix},
\end{equation}
with, for $i=1,2$,
\begin{equation*}
r_{0,i}^E(x-y) =
2\pi\sqrt{E} \,\frac{x_i-y_i}{\|x-y\|}\,  J_1(2\pi\sqrt{E}\|x-y\|).
\end{equation*}
Finally
$$
\Sigma_3^E(x-y) = \begin{pmatrix}
2\pi^2E  &0  &r^E_{1,1}(x-y) &r^E_{1,2}(x-y)\\
0 &2\pi^2 E &r^E_{2,1}(x-y) &r^E_{2,2}(x-y)\\
r^E_{1,1}(x-y) &r^E_{2,1}(x-y) &2\pi^2E &0\\
r^E_{1,2}(x-y) &r^E_{2,2}(x-y) &0 &2\pi^2E
\end{pmatrix},
$$
where for $i=1,2$
\begin{equation}\label{covii}
r^E_{i,i}(x-y)=  2\pi^2 E \left ( J_0(2\pi\sqrt{E}\|x-y\|) 
 + \Big (1 - 2\frac{(x_i - y_i)^2}{\|x-y\|^2}  \Big ) J_2(2\pi\sqrt{E}\|x-y\|) \right),
\end{equation}
and 
\begin{equation}\label{cov12}
r_{12}^E(x-y) = r^E_{2,1}(x-y)= -4\pi^2 E \frac{(x_1 - y_1)(x_2 - y_2)}{\|x - y\|^2} J_2(2\pi\sqrt{E}\|x-y\|).
\end{equation}
\end{lem}

\bigskip 

Let us also define, for $k,l\in \lbrace 0,1,2\rbrace$, 
$$
\widetilde r^E_{k,l}(x,y) = \widetilde r^E_{k,l}(x-y) := \E \left [\widetilde \partial_k B_E(x) \widetilde \partial_l B_E(y) \right ],\qquad x,y\in \R^2,
$$
with $\widetilde \partial_0 B_E := B_E$, where we define the normalized derivatives as 
\begin{equation}\label{normalized}
\widetilde \partial_i := \frac{\partial_i}{\sqrt{2\pi^2E}},\qquad i=1,2\,,
\end{equation}
and accordingly the normalized gradient $\widetilde \nabla$ as 
\begin{equation}\label{normgrad}
\widetilde \nabla := (\widetilde \partial_1, \widetilde \partial_2) = \frac{\nabla}{\sqrt{2\pi^2E}}\,.
\end{equation}

\medskip

One has the following uniform estimate for Bessel functions:
As $\phi \to \infty$, 
\begin{eqnarray}\label{asymp_infinity}
\nonumber r^E((\phi \cos\theta, \phi\sin\theta))  &= & \underbracket[0.5pt][5pt]{\frac{1}{\pi\sqrt{\sqrt E \phi}}\cos\(2\pi \sqrt E\phi-\frac{\pi}{4}\)}_{=: h^E(\theta)g^E(\phi)} + O\left (\frac{1}{E^{3/4}\phi \sqrt \phi }    \right )\\
\nonumber \widetilde r^E_{0,1}((\phi \cos\theta, \phi\sin\theta)) &=& \underbracket[0.5pt][5pt]{\frac{\sqrt{2}\cos\theta}{\pi\sqrt{\sqrt E\phi}}\sin \(2\pi \sqrt E\phi-\frac{\pi}{4} \)}_{=: h^E_{0,1}(\theta)g^E_{0,1}(\phi)} + O\left (\frac{1}{E^{3/4}\phi \sqrt \phi }    \right )\\
\widetilde r^E_{0,2}((\phi \cos\theta, \phi\sin\theta)) &=& \underbracket[0.5pt][5pt]{\frac{\sqrt{2}\sin \theta}{\pi\sqrt{\sqrt E \phi}}\sin\(2\pi \sqrt E\phi-\frac{\pi}{4}\)}_{=: h^E_{0,2}(\theta) g^E_{0,2}(\phi)} + O\left (\frac{1}{E^{3/4}\phi \sqrt \phi }    \right )\\
\nonumber \widetilde r^E_{1,1}((\phi \cos\theta, \phi\sin\theta)) &=& \underbracket[0.5pt][5pt]{\frac{2\cos^2\theta}{\pi\sqrt{\sqrt E\phi}}\cos\(2\pi \sqrt E\phi-\frac{\pi}{4}\)}_{=: h^E_{1,1}(\theta)g^E_{1,1}(\phi)}+ O\left (\frac{1}{E^{3/4}\phi \sqrt \phi }    \right )\\
\nonumber \widetilde r^E_{2,2}((\phi \cos\theta, \phi\sin\theta)) &=& \underbracket[0.5pt][5pt]{\frac{2\sin^2\theta}{\pi\sqrt{\sqrt E\phi}}\cos\(2\pi \sqrt E\phi-\frac{\pi}{4}\)}_{=:h^E_{2,2}(\theta) g^E_{2,2}(\phi)}+ O\left (\frac{1}{E^{3/4}\phi \sqrt \phi }    \right )\\
\nonumber \widetilde r^E_{1,2}((\phi \cos\theta, \phi\sin\theta)) &=&\underbracket[0.5pt][5pt]{\frac{2\cos\theta\sin\theta}{\pi\sqrt{\sqrt E\phi}}\cos\(2\pi \sqrt E\phi-\frac{\pi}{4}\)}_{=: h^E_{1,2}(\theta)g^E_{1,2}(\phi)} + O\left (\frac{1}{E^{3/4}\phi \sqrt \phi }    \right ),
\end{eqnarray}
uniformly on $(\phi, \theta)$, where the constants involved in the $O$-notation do not depend on $E$. 
As $\psi \to 0$,
\begin{flalign}\label{asymp_zero}
&r^1(\psi \cos \theta, \psi \sin \theta)\to 1\,, \qquad \widetilde r^1_{0,i}(\psi \cos \theta, \psi \sin \theta)= O(\psi)\,,\notag\\
&\widetilde r^1_{i,i}(\psi \cos \theta, \psi \sin \theta)\to 1\,,\qquad \widetilde r^1_{1,2}(\psi \cos \theta, \psi \sin \theta)=O(\psi^2)\,,
\end{flalign}
uniformly on $\theta$, for $i=1,2$.

\medskip

\begin{oss}\label{integral_repr}
It is important to stress that the planar random waves can be \emph{formally} represented as a stochastic integral with respect to a Gaussian random measure $W$, in the following way
\begin{equation}\label{rw_SI}
B_E(x)= \int_0^\pi \, f_E(x,t) \, dW(t)= I_1\(f_E(x,\cdot)\)\,,
\end{equation} 
where $f_E$ is chosen in such a way that
\begin{flalign*}
\E \[B_E(x)B_E(y)\] &= J_0(2\pi\sqrt{E}\norm{x-y})\\
&= \int_0^\pi \,\cos\(2\pi\sqrt{E}\norm{x-y}\,\sin t\)\,dt = \int_0^\pi \,f_E(x,t)\,f_E(y,t)\,dt\,.
\end{flalign*}
\end{oss}

\subsection{Chaos}

We refer the reader to \cite[Chapter 2]{NP:12} and \cite[Chapter 5]{PT:10} for a self-contained introduction to Wiener chaos. The next result contains an explicit description of the chaotic expansions of $\L_E(z):=\text{length} (B_E^{-1}(z)\cap D)$ and $\NNN_E(z) := \#\(\(B_E^\C\)^{-1}(z)\cap D\)$, $z\in\R$.

\medskip

\begin{prop} The chaotic expansion of the level curve length in $D$ is 
\begin{equation}\label{expL}
\begin{split}
\L_E(z) = \sum_{q=0}^{+\infty} \L_E^{[q]}(z) 
=&\sqrt{2\pi^2E} \sum_{q=0}^{+\infty} \sum_{u=0}^{q} \sum_{m=0}^{u} \beta_{q - u}(z) \alpha_{m, u - m}\times \cr
&\times \int_{D} H_{q-u}(B_E(x)) H_{m}(\widetilde \partial_1 B_E(x)) H_{u-m}(\widetilde \partial_2 B_E(x))\,dx,
\end{split}
\end{equation}
where $\lbrace \beta_{n}(z)\rbrace_{n\ge 0}$ are the formal coefficients of the chaotic expansion of $\delta_z$ (see Remark \ref{formal_beta}), while $\lbrace \alpha_{n,m}\rbrace_{n,m\ge 0}$ is the sequence of chaotic coefficients of the Euclidean norm in $\R^2$ $\| \cdot \|$ appearing in \cite[Lemma 3.5]{MPRW:16}. Here, the symbol $ \L_E^{[q]}(z) $ indicates the projection of $\L_E(z)$ onto the $q$th Wiener chaos associated with $B_E$, as defined in \cite[Section 2.2]{NP:12}.

For the number of level points in $D$ we have
\begin{equation}\label{expN}
\begin{split}
&\NNN_E(z) = \sum_{q=0}^{+\infty} \NNN_E^{[q]}(z) = 2\pi^2E  \sum_{q=0}^{+\infty}  \sum_{i_1+i_2+i_3+j_1+j_2+j_3=q} \beta_{i_1}(z) \beta_{j_1}(z)\,\zeta_{i_2,i_3,j_2,j_3}\cr
&\int_{D} H_{i_1}(B_E(x)) H_{j_1}(\widehat B_E(x))  H_{i_2}(\widetilde \partial_1 B_E(x)) H_{i_3}(\widetilde \partial_2 B_E(x))H_{j_2}(\widetilde \partial_1 \widehat B_E(x)) H_{j_3}(\widetilde \partial_2 \widehat B_E(x))\,dx,
\end{split}
\end{equation}
where $i_2, i_3, j_2, j_3$ have the same parity; here the sequence $\lbrace \zeta_{i_2,i_3,j_2,j_3}\rbrace$ corresponds to the chaotic expansion of the absolute value of the Jacobian appearing in \cite[Lemma 4.2]{DNPR:19}. Here, the symbol $ \NNN_E^{[q]}(z) $ indicates the projection of $\NNN_E(z)$ onto the $q$th Wiener chaos associated with $B^{\mathbb C}_E$, as defined in \cite[Section 2.2]{NP:12}.
\end{prop}

\medskip

\begin{oss}\label{formal_beta}
The coefficients $\beta_l$ are defined as the limit, as $\eps\rightarrow 0$, of $\beta^\eps_l:=\frac{1}{l!}\eta_l^\eps(z)$, where 
$$
\frac{1}{2\eps}\1_{[z-\eps,z+\eps]}(\cdot)=\sum_{l=0}^\infty\,\frac{1}{l!}\,\eta_l^\eps(z)\,H_l(\cdot)\,.
$$
In \cite[Proposition 7.2.2]{Ro:15}, it is shown that
\begin{flalign}
\eta_{n}(z)&=\lim_{\eps\to 0} \frac{1}{2\eps} \int_{z-\eps}^{z+\eps} \gamma(t) H_{n}(t)    \,dt=\lim_{\eps\to 0}
\frac{1}{2\eps} \int_{z-\eps}^{z+\eps} \gamma(t) (-1)^{n} \gamma^{-1}(t) \frac{d^{n}}{dt^{n}} \gamma(t)  \,dt \notag\\
&=\lim_{\eps\to 0}\frac{(-1)^n}{2\eps} \int_{z-\eps}^{z+\eps} \frac{d^{n}}{dt^{n}} \gamma(t)  \,dt =\gamma(z)\,H_{n}(z).
\end{flalign}
with $\gamma$ the standard Gaussian density on $\R$ and
\begin{equation}  \label{coeff_alfa}
\alpha_{n,n-m}=\frac{1}{2\pi\, (n)!\,(n-m)!} \int_{\R^2} \sqrt{y^2 + z^2} \, H_{n}(y)
H_{n-m}(z) \mathrm{e}^{-\frac{y^2+z^2}{2}}\,dy dz\,,
\end{equation}
where \eqref{coeff_alfa} vanishes whenever $n$ or $n-m$ is odd.
In \cite{DNPR:19}, it is shown that 
$$
\zeta_{a,b,c,d}=\frac{1}{a!\,b!\,c!\,d!\,}\,\,\E\[\abs{XY-ZW}\,H_a(X)H_b(Y)H_c(Z)H_d(W)\]\,,
$$
where $(X, Y, V, W)$ is a standard real four-dimensional Gaussian vector.
\end{oss}

\medskip

In particular, we have
\begin{equation}\label{fewbeta}
\begin{split}
&\beta_0(z)=\gamma(z)H_0(z)=\gamma(z),\quad \beta_1(z)=\gamma(z)H_1(z)=\gamma(z)\,z,\\
&\beta_2(z)=\frac12\,\gamma(z)H_2(z)=\frac12\,\gamma(z)(z^2-1), \quad \beta_3(z)= \frac16\,\gamma(z)H_3(z)=\frac16\,\gamma(z)(z^3-3z), \\  
&\beta_4=\frac1{24}\,\gamma(z)H_4(z)=\frac1{24}\,\gamma(z)(z^4-6z^2+3)\,,
\end{split}
\end{equation}
\begin{flalign}\label{fewalpha}
&\alpha_{0,0}=\frac{\sqrt{2\pi}}2,\quad 
\alpha_{2,0}=\alpha_{0,2}=\frac{\sqrt{2\pi}}8,\quad 
\alpha_{4,0}=\alpha_{0,4}=-\frac{\sqrt{2\pi}}{128},\quad 
\alpha_{2,2}=-\frac{\sqrt{2\pi}}{64}
\end{flalign}
and 
\begin{equation}\label{fewgamma}
\begin{split}
\zeta_{0,0,0,0}&=1,\quad 
\zeta_{2,0,0,0}=\zeta_{0,2,0,0}=\zeta_{0,0,2,0}=\zeta_{0,0,0,2}=\frac14,\\
\zeta_{1,1,1,1}&=-\frac38,\quad 
\zeta_{2,2,0,0}=\zeta_{0,0,2,2}=-\frac1{32},\\
\zeta_{2,0,2,0}&=\zeta_{0,2,0,2}=-\frac1{32},\quad 
\zeta_{2,0,0,2}=\zeta_{0,2,2,0}=\frac{5}{32},\\
\zeta_{4,0,0,0}&=\zeta_{0,4,0,0}=\zeta_{0,0,4,0}=\zeta_{0,0,0,4}=-\frac3{192}.
\end{split}
\end{equation}
Note that, when $z=0$, the odd-chaoses vanish.

\bigskip

Once the chaotic expansions were established, the authors of \cite{NPR:19} proved that, as $E\sto +\infty$ (see \cite[Equation (2.29)]{NPR:19})
\begin{equation*}
\frac{\L_E - \E[\L_E]}{\sqrt{\Var(\L_E)}} = \frac{\L_E^{[4]}}{\sqrt{\Var(\L_E^{[4]})}} + o_{\P}(1), \qquad \frac{\NNN_E - \E[\NNN_E]}{\sqrt{\Var(\NNN_E)}} = \frac{\NNN_E^{[4]}}{\sqrt{\Var(\NNN_E^{[4]})}} + o_{\P}(1)
\end{equation*}
using the following results (and in particular that $\Var \L_E \sim \Var \L_E^{[4]}$).

\medskip

\begin{lem}[{\cite[Lemma 4.1 and 4.2]{NPR:19}}]\label{2_chaos_var}
We have 
\begin{equation}\label{newsecondchaos}
\L_E^{[2]}  = \frac1{8\pi\sqrt{2\,E}}
\int_{\partial D} B_E(x)\langle \nabla B_E(x),n(x)\rangle dx,
\end{equation}
where $n(x)$ is the outward pointing normal at $x$, and hence 
\begin{equation}\label{var2L}
\Var(\L_E^{[2]}) = O(1).
\end{equation}
Moreover, 
\begin{equation}\label{newsecondchaosN}
\NNN_E^{[2]}  = \sqrt{2E}\big(\L_E^{[2]}+\widetilde{\L}_E[2]\big)
\end{equation}
and hence 
\begin{equation}\label{var2N}
\Var(\NNN_E^{[2]}) = O(E).
\end{equation}
\end{lem}

\medskip

\begin{prop}[{\cite[Proposition 6.1]{NPR:19}}]\label{propvar4L} 
The fourth chaotic component of $\L_E$ is given by
\begin{equation}\label{L_4}
\L_E^{[4]}(D)=\frac{\sqrt{2\pi^2\,E}}{128}\{8\,a_{1,E}-a_{2,E}-a_{3,E}-2\,a_{4,E}-8\,a_{5,E}-8\,a_{6,E}\}\,,
\end{equation}
where 
\begin{equation}\label{a's}
\begin{split}
a_{1,E}&:=\int_{D} H_4(B_E(x))dx\,,\quad a_{2,E}:=\int_{D}  H_4(\widetilde{\partial}_1 B_E(x))dx\,,\quad
a_{3,E}:=\int_{D} H_4(\widetilde{\partial}_2 B_E(x))dx\,,\\
a_{4,E} &:= \int_{D}  H_2(\widetilde{\partial}_1 B_E(x))
H_2(\widetilde{\partial}_2 B_E(x))dx\,,\\
a_{5,E} &:= \int_{D}  H_2(B_E(x))
  H_2(\widetilde{\partial}_1 B_E(x))dx\,,\quad
a_{6,E} := \int_{D}  H_2(B_E(x))H_2(\widetilde{\partial}_2 B_E(x))dx\,.
\end{split}
\end{equation}
Its variance satisfies 
\begin{equation}\label{var4L}
\begin{split}
{\Var}(\L_E^{[4]}) &= \frac{\pi^2E}{8192}\,{\Var}\left (8a_{1,E}-a_{2,E}-a_{3,E}-2a_{4,E}-8a_{5,E}-8a_{6,E}
\right )\\
&\sim \frac{{\area}(D)\,\log E}{512\pi},
\end{split}
\end{equation}
where the last asymptotic equivalence holds as $E\to +\infty$. 
\end{prop}

\medskip

\begin{prop}[{\cite[Proposition 6.2]{NPR:19}}]\label{propvar4N} 
The fourth chaotic component of $\NNN_E$ is given by
\begin{equation}\label{N_4}
\NNN_E^{[4]}(D)=a_{E}(D)+\widehat{a}_{E}(D)+b_{E}(D)\,,
\end{equation}
where
$$
a_E(D)=\frac{\pi\,E}{64}\{8\,a_{1,E}(D)-a_{2,E}(D)-2a_{3,E}(D)-8\,a_{4,E}(D)\}\,,
$$
$\widehat{a}_{E}(D)$ is defined in the same way as $a_{E}(D)$, except for the fact that one uses $\widehat{B}_E$ instead of $B_E$, and 
$$
b_E = \frac{\pi E}{8}\{ 2b_{1,E} - b_{2,E} - b_{3,E} -b_{4,E} - b_{5,E} - \frac14 b_{6,E} -\frac14 b_{7,E} + \frac54 b_{8,E} + \frac54 b_{9,E}- 3b_{10,E} \},
$$
with $a_{i,E}$, $i=1,\ldots,4$ defined in \eqref{a's} and 
\begin{equation}\label{b's}\begin{split}
&b_{1,E} := \int_{D}  
H_2(B_E(x))H_2(\widehat{B}_E(x))
dx \qquad b_{2,E}:=\int_{D} H_2(B_E(x)) H_2(\widetilde{\partial}_1 \widehat{B}_E(x)dx\\
&b_{3,E}=\int_{D} H_2(B_E(x)) H_2(\widetilde{\partial}_2 \widehat{B}_E(x))dx\qquad
b_{4,E}=\int_{D} H_2(\widetilde{\partial}_1 B_E(x))H_2(\widehat{B}_E(x)) dx\\
&b_{5,E}:=\int_{D} H_2(\widetilde{\partial}_2 B_E(x))H_2(\widehat{B}_E(x)) dx\qquad
b_{6,E} := \int_{D}  
H_2(\widetilde{\partial}_1 B_E(x))H_2(\widetilde{\partial}_1 \widehat{B}_E(x))
dx\\
&b_{7,E} := \int_{D}  
H_2(\widetilde{\partial}_2 B_E(x))H_2(\widetilde{\partial}_2 \widehat{B}_E(x))
dx\qquad
b_{8,E} := \int_{D}  
H_2(\widetilde{\partial}_1 B_E(x))H_2(\widetilde{\partial}_2 \widehat{B}_E(x))dx\\
&b_{9,E} := \int_{D}  
H_2(\widetilde{\partial}_2 B_E(x))H_2(\widetilde{\partial}_1 \widehat{B}_E(x))dx\\
&b_{10,E} := \int_{D}  \widetilde{\partial}_1 B_E(x)\widetilde{\partial}_2 B_E(x)
\widetilde{\partial}_1 \widehat{B}_E(x)\widetilde{\partial}_2 \widehat{B}_E(x)dx.
\end{split}\end{equation}
Its variance satisfies 
\begin{equation}\label{eqbella}
{\Var}(\NNN_E^{[4]})=2{\Var}(a_E)+{\Var}(b_E)\sim 
\frac{11{\area}(D)}{32\pi}\,E\log E,
\end{equation}
where the last asymptotic equivalence holds as $E\to +\infty$. 
\end{prop}

\subsection*{Acknowledgments}

The authors would like to thank Maurizia Rossi for insightful discussions. AV would like also to thank Guangqu Zheng for some useful comments on an early version of this work. Part of this work has been written in the framework of AFR research project \emph{High-dimensional challenges and non-polynomial transformations in probabilistic approximations} (HIGH-NPOL) funded by FNR -- Luxembourg National Research Fund. Giovanni Peccati
is also supported by the FNR grant FoRGES (R-AGR-3376-10) at Luxembourg University.

\bibliographystyle{alpha}
\bibliography{Bibliography}

\begin{thebibliography}{MPRW16}

\bibitem[APP18]{APP:17}
J.~Angst, V.~H. Pham, and G.~Poly.
\newblock Universality of the nodal length of bivariate random trigonometric
  polynomials.
\newblock {\em Transactions of the A.M.S.}, 370:8331--8357, 2018.

\bibitem[BCW17]{BCW:17}
D.~Beliaev, V.~Cammarota, and I.~Wigman.
\newblock Two point function for critical points of a random plane wave.
\newblock {\em Int. Math. Res. Notices}, page rnx197, 2017.

\bibitem[Ber77]{Be:77}
M.V. Berry.
\newblock Regular and irregular semiclassical wavefunctions.
\newblock {\em J. Phys. A}, 10(12):2083--2092, 1977.

\bibitem[Ber02]{Be:02}
M.V. Berry.
\newblock Statistics of nodal lines and points in chaotic quantum billiards:
  perimeter corrections, fluctuations, curvature.
\newblock {\em J. Phys. A}, 35(13):3025--3038, 2002.

\bibitem[Bil99]{B:99}
P.~Billingsley.
\newblock {\em Convergence of Probability Measures}.
\newblock John Wiley and Sons, second edition, 1999.

\bibitem[BM17]{BM:17}
J.~Benatar and R.~W. Maffucci.
\newblock Random waves on $\mathbb t^3$: Nodal area variance and lattice point
  correlations.
\newblock {\em Int. Math. Res. Notices}, page rnx220, 2017.

\bibitem[BMW19]{BMW:19}
J.~Benatar, D.~Marinucci, and I.~Wigman.
\newblock Planck-scale distribution of nodal length of arithmetic random waves.
\newblock {\em Journal d'Analyse Math\'ematique (in press)}, 2019+.

\bibitem[BP85]{BR:85}
R.~F. Bass and R.~Pyke.
\newblock The space $d(a)$ and weak convergence for set-indexed processes.
\newblock {\em Ann. Probab.}, 13(3):860--884, 1985.

\bibitem[BW18]{BW:18}
D.~Beliaev and I.~Wigman.
\newblock Volume distribution of nodal domains of random band-limited
  functions.
\newblock {\em Probab. Theory Relat. Fields}, 172(1):453--492, 2018.

\bibitem[Cam19]{Ca:19}
V.~Cammarota.
\newblock Nodal area distribution for arithmetic random waves.
\newblock {\em Transactions of the A. M. S. (in press)}, 2019+.

\bibitem[CH16]{CH:16}
H.~Canzani and B.~Hanin.
\newblock Local universality for zeros and critical points of monochromatic
  random waves.
\newblock {\tt ar{X}iv:610.09438}, 2016.

\bibitem[CM18]{CM:18b}
V.~Cammarota and D.~Marinucci.
\newblock A quantitative central limit theorem for the euler-poincar{\`e}
  characteristic of random spherical eigenfunctions.
\newblock {\em Ann. Probab.}, 46(6):3188--3228, 2018.

\bibitem[CM19a]{CM:19}
V.~Cammarota and D.~Marinucci.
\newblock On the correlation of critical points and angular trispectrum for
  random spherical harmonics.
\newblock {\tt ar{X}iv:1907.05810}, 2019.

\bibitem[CM19b]{CM:19b}
V.~Cammarota and D.~Marinucci.
\newblock A reduction principle for the critical values of random spherical
  harmonics.
\newblock {\em Stochastic Process. Appl. (in press)}, 2019+.

\bibitem[CMW16a]{CMW:16}
V.~Cammarota, D.~Marinucci, and I.~Wigman.
\newblock Fluctuations of the euler-poincar\'e characteristic for random
  spherical harmonics.
\newblock {\em Proc. Amer. Math. Soc.}, 144:4759--4775, 2016.

\bibitem[CMW16b]{CMW:16b}
V.~Cammarota, D.~Marinucci, and I.~Wigman.
\newblock On the distribution of the critical values of random spherical
  harmonics.
\newblock {\em The Journal of Geometric Analysis}, 26:3252--3324, 2016.

\bibitem[CS19]{CS:19}
H.~Canzani and P.~Sarnak.
\newblock Topology and nesting of the zero set components of monochromatic
  random waves.
\newblock {\em Comm. Pure Appl. Math.}, 72(2):343--374, 2019.

\bibitem[DNPR19]{DNPR:19}
F.~Dalmao, I.~Nourdin, G.~Peccati, and M.~Rossi.
\newblock Phase singularities in complex arithmetic random waves.
\newblock {\em Electron. J. Probab.}, 24:45 pp., 2019.

\bibitem[DOP16]{DOP:16}
M.~R. Dennis, K.~O'Holleran, and M.J. Padgett.
\newblock Singular optics: optical vortices and polarization singularities.
\newblock {\em Progress in Optics}, 53:293--363, 2016.

\bibitem[Kee19]{K:19}
Blake Keeler.
\newblock A logarithmic improvement in the two-point {W}eyl law for manifolds
  without conjugate points.
\newblock {\tt ar{X}iv:1905.05136}, 2019.

\bibitem[KKW13]{KKW:13}
M.~Krishnapur, P.~Kurlberg, and I.~Wigman.
\newblock Non-universality of nodal length distribution for arithmetic random
  waves.
\newblock {\em Ann. Math.}, 177(2):699--737, 2013.

\bibitem[KW18]{KW:18}
P.~Kurlberg and I.~Wigman.
\newblock Non-universality of the nazarov-sodin constant for random plane waves
  and arithmetic random waves.
\newblock {\em Advances in Mathematics}, 330:516--552, 2018.

\bibitem[MP11]{MP:11}
D.~Marinucci and G.~Peccati.
\newblock {\em Random Fields on the Sphere: Representation, Limit Theorems and
  Cosmological Applications}.
\newblock London Mathematical Society Lecture Note Series. Cambridge University
  Press, 2011.

\bibitem[MPRW16]{MPRW:16}
D.~Marinucci, G.~Peccati, M.~Rossi, and I.~Wigman.
\newblock Non-universality of nodal length distribution for arithmetic random
  waves.
\newblock {\em GAFA}, 3:926--960, 2016.

\bibitem[MRW19]{MRW:19}
D.~Marinucci, M.~Rossi, and I.~Wigman.
\newblock The asymptotic equivalence of the sample trispectrum and the nodal
  length for random spherical harmonics.
\newblock {\em Ann. Inst. Henri Poincar\'e Probab. Stat. (in press)}, 2019+.

\bibitem[Neu71]{Ne:71}
Georg Neuhaus.
\newblock On weak convergence of stochastic processes with multidimensional
  time parameter.
\newblock {\em Ann. Math. Statist.}, 42(4):1285--1295, 08 1971.

\bibitem[NP12]{NP:12}
I.~Nourdin and G.~Peccati.
\newblock {\em Normal Approximation with Malliavin Calculus: From Stein's
  Method to Universality}.
\newblock Cambridge University Press, 2012.

\bibitem[NPR19]{NPR:19}
I.~Nourdin, G.~Peccati, and M.~Rossi.
\newblock Nodal statistics of planar random waves.
\newblock {\em Comm. Math. Phys.}, 369(1):99--151, 2019.

\bibitem[NS16]{NS:16}
F.~Nazarov and M.~Sodin.
\newblock Asymptotic laws for the spatial distribution and the number of
  connected components of zero sets of gaussian random functions.
\newblock {\em J. Math. Phys. Anal. Geom.}, 12(3):205--278, 2016.

\bibitem[ORW08]{ORW:08}
F.~Oravecz, D.~Rudnick, and I.~Wigman.
\newblock The {L}eray measure of nodal sets for random eigenfunctions on the
  torus.
\newblock {\em Annales de l'institut Fourier}, 58(1):299--335, 2008.

\bibitem[PR18]{PR:18}
G.~Peccati and M.~Rossi.
\newblock Quantitative limit theorems for local functionals of arithmetic
  random waves.
\newblock Abel Symposium 2016 (Springer), 2018.

\bibitem[PT05]{PT:05}
G.~Peccati and C.~A. Tudor.
\newblock Gaussian limits for vector-valued multiple stochastic integrals.
\newblock {\em S\'eminaire de Probabilit\'es}, XXXVIII:247--262, 2005.

\bibitem[PT10]{PT:10}
G.~Peccati and M.~S. Taqqu.
\newblock {\em Wiener Chaos: Moments, Cumulants and Diagrams}.
\newblock Springer-Verlag, 2010.

\bibitem[Ros15]{Ro:15}
M.~Rossi.
\newblock On the high energy behavior of nonlinear functionals of random
  eigenfunctions on $\mathbb{S}^d$.
\newblock In {\em 19th European Young Statisticians Meeting in Prague}, pages
  119--124, 2015.

\bibitem[RS01]{RS:01}
Alexander Rusakov and Oleg Seleznjev.
\newblock On weak convergence of functionals on smooth random functions.
\newblock {\em Mathematical Communications}, 6:123--134, 2001.

\bibitem[RW08]{RW:08}
Z.~Rudnick and I.~Wigman.
\newblock On the volume of nodal sets for eigenfunctions of the {L}aplacian on
  the torus.
\newblock {\em Ann. Henri Poincar{\'e}}, 9(1):109--130, 2008.

\bibitem[RW18]{RW:18}
M.~Rossi and I.~Wigman.
\newblock Asymptotic distribution of nodal intersections for arithmetic random
  waves.
\newblock {\em Nonlinearity}, 31(10):4472, 2018.

\bibitem[RY99]{RY:99}
D.~Revuz and M.~Yor.
\newblock {\em Continuous Martingales and Brownian Motion}.
\newblock Springer-Verlag, 1999.

\bibitem[SW16]{SW:16}
P.~Sarnak and I.~Wigman.
\newblock Topologies of nodal sets of random band limited functions.
\newblock In {\em Advances in the Theory of Automorphic Forms and Their
  $L$-functions}, volume 664 of {\em Contemporary Mathematics}, 2016.

\bibitem[Tod18]{To:18b}
A.~P. Todino.
\newblock Nodal lengths in shrinking domains for random eigenfunctions on
  $\mathbb{S}^2$.
\newblock {\tt ar{X}iv:1807.11787}, 2018.

\bibitem[Tod19]{To:18}
A.~P. Todino.
\newblock A quantitative central limit theorem for the excursion area of random
  spherical harmonics over subdomains of $\mathbb{S}^2$.
\newblock {\em J. Math. Phys.}, 60(2), 2019.

\bibitem[UR13]{UR:13}
J.~Urbina and K.~Richter.
\newblock Random quantum states: recent developments and applications.
\newblock {\em Advances in Physics}, 62:787--831, 2013.

\bibitem[Wig10]{Wi:10}
I.~Wigman.
\newblock Fluctuations of the nodal length of random spherical harmonics.
\newblock {\em Comm. Math. Phys.}, 298(3):787--831, 2010.

\bibitem[Zel09]{Ze:09}
S.~Zelditch.
\newblock Real and complex zeros of riemannian random waves.
\newblock In {\em Spectral Analysis in Geometry and Number Theory}, volume 484
  of {\em Contemporary Mathematics}, 2009.

\end{thebibliography}

\end{document}